\documentclass[11pt]{amsart}

\usepackage[T1]{fontenc}
\usepackage{lmodern}
\usepackage{amsmath,amssymb,mathtools}
\usepackage{microtype}
\usepackage[colorlinks=true,linkcolor=blue,citecolor=blue,urlcolor=blue,filecolor=blue]{hyperref}
\usepackage[nameinlink,noabbrev]{cleveref}

\newtheorem{theorem}{Theorem}[section]
\newtheorem{proposition}[theorem]{Proposition}
\newtheorem{lemma}[theorem]{Lemma}
\newtheorem{corollary}[theorem]{Corollary}
\theoremstyle{remark}
\newtheorem{remark}[theorem]{Remark}

\numberwithin{equation}{section}

\title[Subregular affine cells for $L_{-1}(D_\ell)$]{Subregular affine cells and the level $-1$ vertex algebra of type $D$}
\author{Sihai Jin}
\address{Department of Mathematics, Sichuan University, China}
\email{jinsihai@stu.scu.edu.cn}
\date{}

\begin{document}

\begin{abstract}
We prove the simple-object prediction of Shan--Yan--Zhao and a basis-preserving dual-cell realization for the distinguished vacuum block of the simple affine vertex algebras $L_{-1}(D_\ell)$, $\ell\ge5$.  The block has exactly $\ell+1$ simple objects, indexed by the subregular affine left cell containing $s_0$.  The proof combines a primitive-ideal inclusion, an independent exhaustion argument, and a finite-length step.  A noncritical Sugawara lift supplies finite-dimensional weight-space detectors in the original Shan--Yan--Zhao category-$\mathcal O$ block, so d\'evissage applies to its ordinary Grothendieck group.  We then identify this group, basis by basis, with the $q=1$ specialization of the corresponding dual affine left-cell module.  An injective signed normalized-character realization identifies the resulting image with the canonical dual-cell image in the completed singular-orbit module and hence supplies the corresponding abstract $\widehat W$-module structure.  We do not identify this action with a functorial action arising from affine twisting functors or Kashiwara--Tanisaki localization.  The subregular inverse Kazhdan--Lusztig calculation of Bezrukavnikov--Kac--Krylov also yields uniform character formulas.
\end{abstract}

\maketitle

\section{Introduction and main results}

Let $L_k(\mathfrak g)$ denote the simple affine vertex algebra associated
with a simple complex Lie algebra $\mathfrak g$.  Shan--Yan--Zhao predict
that certain distinguished blocks of $L_k(\mathfrak g)$-modules are
controlled by affine Kazhdan--Lusztig left cells; in particular, the
Grothendieck group of such a block should be the $q=1$ specialization of
the corresponding dual affine left-cell module \cite{ShanYanZhao2026}.  The
case treated here is
$$
   (\mathfrak g,k)=(D_\ell,-1),\qquad \ell\ge5.
$$

We use the Bourbaki realization
$$
 \alpha_i=\varepsilon_i-\varepsilon_{i+1}\ (1\le i\le\ell-1),
 \qquad \alpha_\ell=\varepsilon_{\ell-1}+\varepsilon_\ell,
$$
with $\theta=\varepsilon_1+\varepsilon_2$ and
$\rho=(\ell-1,\ell-2,\ldots,1,0)$.  Put
$$
   N:=k+h^\vee=2\ell-3,\qquad r:=\ell-2,
$$
and let $s_0,s_1,\ldots,s_\ell$ be the simple reflections of
$D_\ell^{(1)}$.  For $1\le i\le\ell$, let $z_i$ be the product of
finite simple reflections along the unique Dynkin path from $i$ to $2$,
ending in $s_2$, and set
$$
   w_0=s_0,\qquad w_i=z_i s_0,\qquad
   \mu_0=0,\qquad \mu_i=z_i\rho-\rho.
$$
The subregular-cell description of
Bezrukavnikov--Kac--Krylov \cite{BezrukavnikovKacKrylov2024} gives
$$
   \mathbf c^L(s_0)=\{w_0,w_1,\ldots,w_\ell\}.
$$
A direct dot-action calculation, recorded in
Proposition~\ref{prop:path-dot-action}, gives
$$
   w_i\circ(-\Lambda_0)=-\Lambda_0+\mu_i.
$$

\begin{theorem}[Main theorem]
For every $\ell\ge5$,
\begin{equation}
   \operatorname{Irr}\mathcal O_{-\Lambda_0}
   \!\left(L_{-1}(D_\ell)\right)
   =
   \left\{
      L\!\left(w_i\circ(-\Lambda_0)\right)
      \;\middle|\;0\le i\le\ell
   \right\}.
\end{equation}
Moreover, the natural exact inclusion of the vertex-algebra block into the
ambient affine category-$\mathcal O$ block induces an injective map on
ordinary Grothendieck groups, and after the specialization
$v=q^{1/2}\mapsto1$ there is a unique basis-preserving isomorphism
\begin{equation}
   K_0\mathcal O_{-\Lambda_0}\!\left(L_{-1}(D_\ell)\right)
   \xrightarrow{\ \sim\ }
   \left.
      \mathcal H^\vee_{\mathrm{aff},\mathbf c^L(s_0)}
   \right|_{v=1},
   \qquad
   [L(w_i\circ(-\Lambda_0))]\longmapsto D_{w_i}.
\end{equation}
The signed normalized-character map sends this basis to
$\varepsilon(w_i)R_{\widehat{\mathfrak g}}\operatorname{ch}L(w_i\circ(-\Lambda_0))$
and identifies its image with the canonical dual-cell image in the completed
singular-orbit module.  Pulling back the canonical cell action through this
injective character realization gives an abstract $\widehat W$-module
structure for which the displayed map is equivariant.  This proves the
simple/cell basis prediction and the underlying dual-cell realization in
\cite[Conjecture~5.2.1]{ShanYanZhao2026}.  We do not claim, without an
additional comparison theorem, that this action coincides with the functorial
action anticipated there from affine twisting functors or
Kashiwara--Tanisaki localization.  The $\ell+1$ simple characters admit a
uniform formula obtained from the subregular inverse Kazhdan--Lusztig
coefficients of \cite{BezrukavnikovKacKrylov2024}; see
Corollary~\ref{cor:uniform-characters}.
\end{theorem}

\subsection*{Relation to previous work and scope of the revision}
The construction of the candidate quotient $Q_\ell$, the computation of the
minimal-reduction images of its defining singular vectors, the weighted Casimir
identity, and the maximality statement
$Q_\ell\cong L_{-1}(D_\ell)$ are taken from \cite{Jin2026}.  They are used here
as inputs to a different problem: determining the distinguished vacuum block,
proving the common primitive-ideal membership and an independent exhaustion
statement, and identifying the ordinary Grothendieck group with the dual
subregular left-cell module.  In particular, the membership and exhaustion
arguments classify the vacuum block of the candidate quotient before the
maximality theorem is invoked.

An earlier version of this paper passed from the classification of simple
objects to the ordinary Grothendieck group without first proving a finite-length
statement.  Jain \cite[Sections~2--5]{Jain2026} identified this gap, exhibited a
counterexample to the formal categorical implication, and proved finite length
and the cell realization for a grading-restricted vacuum block by means of
finite-dimensional detectors.  The present revision uses that detector idea in
the original Shan--Yan--Zhao category.  The additional input is the canonical
noncritical lift of Campbell--Dhillon--Raskin
\cite[Remark~3.5.3]{CampbellDhillonRaskin2021}, whose finite-dimensional
extended weight spaces give exact detectors.  After d\'evissage, injectivity of
the ambient Grothendieck-group map is proved directly from the linear
independence of the formal characters of the finitely many simple modules;
it is not inferred from fullness of the subcategory and does not use the
ambient $W_{\mathrm{aff}}$-linearity assertion of \cite[Section~5.2]{ShanYanZhao2026}.
The dual-cell module is then recovered independently from the signed
normalized-character realization.  Thus the correction concerns the
Grothendieck-group step and does not alter the independent membership and
exhaustion arguments.

For orientation, Section~\ref{sec:subregular-primitive} identifies the relevant
affine cell and the common primitive annihilator.  Section~\ref{sec:candidate-zhu-minimalW}
collects the candidate-quotient and finite-BRST input, Section~\ref{sec:membership}
proves the membership theorem, Section~\ref{sec:exhaustion} proves exhaustion,
and Section~\ref{sec:final-left-cell} establishes finite length, the
basis-preserving signed-character/cell realization of the Grothendieck group,
and the character formulas.  The
appendices contain the type-$D$ calculations and the rank-five boundary case.

We work first with the candidate quotient
$$
   Q_\ell
   :=V^{-1}(D_\ell)/\langle\sigma(w_A)^r,\sigma(w_D)^2\rangle,
   \qquad
   A(Q_\ell)\cong U(D_\ell)/I_\ell^{\mathrm{cand}}.
$$
The identification $Q_\ell\cong L_{-1}(D_\ell)$ from \cite{Jin2026}
is invoked only after its vacuum block has been classified.
Minimal Drinfeld--Sokolov reduction gives a lisse quotient whose Ramond Zhu
algebra $B_\ell$ carries the natural current-zero-mode action of
$A_1\oplus D_r$ and satisfies
\begin{equation}
\label{eq:intro-nilpotence}
   e_A^r=0,\qquad e_D^2=0.
\end{equation}
Together with a weighted Casimir trace identity, these relations provide the
finite-dimensional bounds used below.

Set
$$
   J_{2,\ell}:=
   \operatorname{Ann}_{U(D_\ell)}L_{D_\ell}(-\alpha_2).
$$
All non-vacuum path tops have annihilator $J_{2,\ell}$, so their descent
reduces to the inclusion
\begin{equation}
\label{eq:intro-membership}
   I_\ell^{\mathrm{cand}}\subset J_{2,\ell}.
\end{equation}
A Li twist by $\varpi^\vee+\eta_1^\vee$, together with the
Premet--Whittaker comparison, Chen's annihilator theorem, and quotient
Skryabin equivalence, proves \eqref{eq:intro-membership}.  Independently,
filtered Zhu theory reduces the
exhaustion problem to the zero and minimal nilpotent orbits; the finite-type
bounds and the type-$D$ lattice gap then give
\begin{equation}
\label{eq:intro-exhaustion}
   \mu\in\{\mu_0,\mu_1,\ldots,\mu_\ell\}.
\end{equation}
These two statements classify the candidate quotient.  A separate
finite-length argument then permits d\'evissage of the ordinary
Grothendieck group.  The simple basis is identified with the canonical
left-cell basis, and the injective signed normalized-character realization
places this identification inside the completed singular-orbit module without
using the deferred functorial ambient Hecke action.  The inverse
Kazhdan--Lusztig calculation then gives the character formulas.

Appendix~\ref{app:type-D-calculations} contains the type-$D$ root
computations used in the finite-$W$ comparison.  The rank-five case, where
$D_3\cong A_3$, is recorded separately in
Appendix~\ref{app:rank-five-boundary}.

\section{Subregular affine cells and primitive ideals}
\label{sec:subregular-primitive}

Retain the notation of the introduction.  Let $\widehat W$ be the affine
Weyl group and write
$w\circ\lambda=w(\lambda+\widehat\rho)-\widehat\rho$.

\subsection{The distinguished subregular left cell}

For the shifted vacuum
$$
   A_\ell:=-\Lambda_0+\widehat\rho=N\Lambda_0+\rho
$$
one has
$$
   \langle A_\ell,\alpha_0^\vee\rangle=0,
   \qquad
   \langle A_\ell,\alpha_i^\vee\rangle=1\quad(1\le i\le\ell),
$$
so $\operatorname{Stab}_{\widehat W}(A_\ell)=\langle s_0\rangle$.  The
distinguished Shan--Yan--Zhao element is therefore $s_0$.  In the subregular
cell, unique reduced words ending in a fixed simple reflection form one left
cell \cite[Proposition~5.1]{BezrukavnikovKacKrylov2024}
\cite[Proposition~3.6]{Xu2019}.  Thus $\mathbf c^L(s_0)$ consists of the
unique reduced paths in the affine Dynkin tree ending at node $0$.

For $1\le i\le\ell$, let $z_i$ be the finite path word from node $i$
to node $2$, ending in $s_2$; explicitly
\begin{align*}
 z_1&=s_1s_2, & z_2&=s_2,\\
 z_i&=s_i s_{i-1}\cdots s_3s_2 &&(3\le i\le\ell-2),\\
 z_{\ell-1}&=s_{\ell-1}s_{\ell-2}\cdots s_3s_2,
 &z_\ell&=s_\ell s_{\ell-2}\cdots s_3s_2.
\end{align*}
Set $w_0=s_0$ and $w_i=z_i s_0$.

\begin{proposition}
\label{prop:subregular-left-cell}
$$
   \mathbf c^L(s_0)=\{w_0,w_1,\ldots,w_\ell\}.
$$
\end{proposition}

\begin{proof}
The affine $D_\ell^{(1)}$-diagram is a tree, hence there is exactly one
reduced path from each vertex to $0$; these are precisely the words above.
\end{proof}

\begin{proposition}
\label{prop:path-dot-action}
Put $\mu_0=0$ and $\mu_i=z_i\rho-\rho$ for $i\ge1$.  Then
$$
   w_i\circ(-\Lambda_0)=-\Lambda_0+\mu_i
   \qquad(0\le i\le\ell).
$$
\end{proposition}

\begin{proof}
Since $s_0A_\ell=A_\ell$ and every $z_i$ fixes $\Lambda_0$,
$z_i s_0A_\ell-\widehat\rho=-\Lambda_0+z_i\rho-\rho$.
\end{proof}

The path weights in orthogonal coordinates are
\begin{align}
 \mu_1&=(-2,1,1,0,\ldots,0),
 &\mu_2&=(0,-1,1,0,\ldots,0),\\
 \mu_i&=-\varepsilon_2-\cdots-\varepsilon_i+(i-1)\varepsilon_{i+1}
 &&(3\le i\le\ell-2),\\
 \mu_{\ell-1}&=-\varepsilon_2-\cdots-\varepsilon_{\ell-1}
                 +(\ell-2)\varepsilon_\ell,\\
 \mu_\ell&=-\varepsilon_2-\cdots-\varepsilon_{\ell-1}
                 -(\ell-2)\varepsilon_\ell.
\end{align}
These follow directly by applying the path reflections to $\rho$.

\subsection{A common primitive annihilator}
\label{subsec:primitive-compression}

For a finite highest weight $\lambda$, write
$$
   J(\lambda):=\operatorname{Ann}_{U(\mathfrak g)}L_{\mathfrak g}(\lambda),
   \qquad
   J_{2,\ell}:=J(-\alpha_2).
$$

\begin{theorem}
\label{thm:primitive-annihilator-compression}
For every $1\le i\le\ell$,
$$
   J(\mu_i)=J_{2,\ell}.
$$
\end{theorem}

\begin{proof}
Each $z_i$ is the unique reduced path ending in $s_2$, hence
$z_i\sim_Ls_2$ by \cite[Proposition~3.6]{Xu2019}.  The Duflo--Joseph
theorem for the regular integral block gives
$$
 \operatorname{Ann}L_{\mathfrak g}(x\rho-\rho)
 =\operatorname{Ann}L_{\mathfrak g}(y\rho-\rho)
 \iff x\sim_Ly
$$
(see \cite[Theorem~2.3.3(i)]{ShanYanZhao2026} for these conventions).
Since $\mu_i=z_i\rho-\rho$ and $s_2\rho-\rho=-\alpha_2$, this gives
$J(\mu_i)=J_{2,\ell}$.
\end{proof}

\begin{corollary}
\label{cor:single-membership-test}
Let $Q$ be a positive-energy quotient of $V^{-1}(D_\ell)$ with
$A(Q)\cong U(D_\ell)/I$.  For every $1\le i\le\ell$,
$$
   L(-\Lambda_0+\mu_i)\text{ factors through }Q
   \quad\Longleftrightarrow\quad
   I\subset J_{2,\ell}.
$$
\end{corollary}

\begin{proof}
By Theorem~\ref{thm:primitive-annihilator-compression}, the assertion
reduces to Zhu descent.  If $I\subset J_{2,\ell}$, the simple top
$L_{\mathfrak g}(\mu_i)$ is an $A(Q)$-module.  The $n=0$ simple-object correspondence in Zhu theory
\cite[Theorem~4.9]{DongLiMason1998} gives a unique irreducible admissible
$Q$-module whose degree-zero space is this simple $A(Q)$-module.  Only
this bijection on simple objects is used here; no equivalence of the full module
categories is being asserted.  Inflate the resulting module along
$V^{-1}(D_\ell)\twoheadrightarrow Q$.  The inflated module is irreducible,
is generated by the same affine highest-weight line, and has affine highest
weight $-\Lambda_0+\mu_i$.  The universal affine highest-weight module of
that weight has a unique irreducible quotient, namely
$L(-\Lambda_0+\mu_i)$.  Hence the inflated Zhu lift is precisely this
module, so it factors through $Q$.  The converse follows by taking the
degree-zero top.
\end{proof}

\begin{remark}
Direct minimal reduction does not detect the non-vacuum descent.  Indeed,
$z_i^{-1}\theta=s_2\theta=\theta-\alpha_2$, so
$\langle\mu_i,\theta^\vee\rangle=-1$ and hence
$$
   \langle-\Lambda_0+\mu_i,\alpha_0^\vee\rangle=0.
$$
Arakawa's irreducible minimal-reduction theorem
\cite[Theorem~6.7.4]{Arakawa2005} therefore gives
$$
   H^0_{\mathrm{DS},f_\theta}\bigl(L(-\Lambda_0+\mu_i)\bigr)=0
   \qquad(1\le i\le\ell),
$$
which motivates the finite-$W$ argument of Section~\ref{sec:membership}.
\end{remark}
\section{The candidate Zhu quotient and the minimal W-algebra}
\label{sec:candidate-zhu-minimalW}

Let $\mathfrak g=D_\ell$, $r=\ell-2$, $N=2r+1$, and
$K=r-1$.  In this section all constructions are made for the candidate
quotient; no simplicity statement for $Q_\ell$ is used.

\subsection{The candidate quotient and its reduction}

Let $w_A,w_D$ be highest vectors in the two Arakawa--Moreau summands of
$S^2(\mathfrak g)$, and put
$$
   s_A=\sigma(w_A)^r,
   \qquad
   s_D=\sigma(w_D)^2.
$$
These are affine singular at level $-1$
\cite[Theorem~4.2]{ArakawaMoreau2018}.  Define
\begin{equation}
   Q_\ell:=V^{-1}(\mathfrak g)/\langle s_A,s_D\rangle,
   \qquad
   A(Q_\ell)\cong U(\mathfrak g)/I_\ell^{\mathrm{cand}},
\end{equation}
where $I_\ell^{\mathrm{cand}}=\langle[s_A],[s_D]\rangle$.

For the minimal nilpotent $f_\theta=e_{-\theta}$,
$$
   \mathfrak g^\natural\cong A_1\oplus D_r,
   \qquad
   \mathfrak g_{-1/2}
   \cong L_{A_1}(\varpi)\boxtimes L_{D_r}(\eta_1),
$$
and the two current levels are $K=r-1$ and $1$.  Let $J_A,J_D$
be highest-root currents.  Jin's \cite[Lemma~3.7 and Equation~(37)]{Jin2026}
place the defining sequence in the exactness category and show that the reduced
singular submodules are cyclic.  With $c_A,c_D\in\mathbb C^\times$,
\cite[Lemma~3.3 and Proposition~3.6]{Jin2026} give
$$
   [s_A]_{\mathrm{DS}}=c_A:J_A^r:,
   \qquad
   [s_D]_{\mathrm{DS}}=c_D:J_D^2:.
$$
The generated-quotient lemma \cite[Proposition~2.3]{Jin2026}, or directly
\cite[Proposition~3.8]{Jin2026}, gives
\begin{equation}
\label{eq:exact-DS-Q}
   H^0_{\mathrm{DS},f_\theta}(Q_\ell)
   \cong
   \widetilde{\mathcal W}_\ell
   :=\mathcal W^{-1}(\mathfrak g,f_\theta)/
     \langle:J_A^r:,:J_D^2:\rangle.
\end{equation}
For the present level, $Q_\ell$ is the quotient $\widetilde V_{k_1}$
with $k_1=-1$ in \cite[Section~5]{ArakawaMoreau2018}.  Hence
\cite[Proposition~5.2]{ArakawaMoreau2018} gives
$X_{Q_\ell}=\overline{\mathcal O_{\min}}$.  The associated-variety
formula for minimal reduction therefore shows that
$\widetilde{\mathcal W}_\ell$ is nonzero and lisse.

Set
$$
   B_\ell:=A_{\mathrm R}(\widetilde{\mathcal W}_\ell).
$$
It is a nonzero finite-dimensional algebra
\cite[Section~3.3.1]{Jin2026}.  Writing
$e_A=[J_A]_{\mathrm R}$ and $e_D=[J_D]_{\mathrm R}$, the Ramond Zhu
product of mutually isotropic weight-one currents gives
\begin{equation}
   e_A^r=0,
   \qquad
   e_D^2=0.
\end{equation}

\begin{lemma}[Finite type bounds]
\label{lem:allowed-finite-types}
Every irreducible $A_1\oplus D_r$-constituent of a finite-dimensional
$B_\ell$-module is
$$
   L_{A_1}(a\varpi)\boxtimes L_{D_r}(\lambda_D),
$$
where
\begin{equation}
   0\le a\le r-1,
   \qquad
   \langle\lambda_D,\theta_D^\vee\rangle\le1.
\end{equation}
\end{lemma}

\begin{proof}
Let $v$ be a highest vector of an irreducible constituent.  In the
highest-root $\mathfrak{sl}_2$-strings,
$e_A^r f_A^r v\ne0$ if $a\ge r$, while
$e_D^2 f_D^2 v\ne0$ if
$\langle\lambda_D,\theta_D^\vee\rangle\ge2$.  These contradict
$e_A^r=0$ and $e_D^2=0$, respectively.
\end{proof}

We normalize quadratic Casimirs by
$$
   c_2(a\varpi)=\frac{a(a+2)}2,
   \qquad
   c_2(\eta_1)=2r-1.
$$

\subsection{The weighted Casimir identity}

In the present Casimir normalization, Jin's level $-1$ Ramond
relation \cite[Equation~(27)]{Jin2026} and contraction calculation
\cite[Lemma~3.11]{Jin2026} apply verbatim; in particular the mixed
$A_1$--$D_r$ contribution has trace zero and the two diagonal
contraction coefficients are $1$ and $2/r$.

\begin{proposition}[Weighted Casimir trace identity]
Let $M\ne0$ be a finite-dimensional $B_\ell$-module on which
$L=[\omega]_{\mathrm R}$ acts by $h$.  Then
\begin{equation}
\label{eq:weighted-Casimir}
   \frac{\operatorname{tr}_M\Omega_A}{\dim M}
   +
   \frac{r+2}{2r}
   \frac{\operatorname{tr}_M\Omega_D}{\dim M}
   =
   Nh+\frac{r-1}{4}.
\end{equation}
\end{proposition}

\begin{proof}
Put $d=\dim M$ and $\tau_*=\operatorname{tr}_M\Omega_*$.  Tracing
\cite[Equation~(27)]{Jin2026}, using
\cite[Lemma~3.11]{Jin2026}, and choosing
$\langle u,v\rangle\ne0$ gives
$$
   2\tau_A+\left(1+\frac2r\right)\tau_D
   =2Nhd+\frac{r-1}{2}d.
$$
Dividing by $2d$ yields \eqref{eq:weighted-Casimir}, in agreement with
\cite[Proposition~3.12]{Jin2026}.
\end{proof}

\subsection{Finite BRST and filtered Zhu}

Let $H=U(\mathfrak g,f_\theta)$.  Applying Arakawa's Zhu--DS
compatibility for quotients \cite[Theorems~8.1 and~8.5]{Arakawa2015} to
\eqref{eq:exact-DS-Q} gives the finite realization needed later.  We stress a
convention that is important here: throughout the relevant part of
\cite{Arakawa2015}, $A(V)$ denotes the $L_0$-twisted Zhu algebra in
the sense of De Sole--Kac (see the convention stated at the beginning of that
paper and Section~4 there).  Thus, for the half-integrally graded minimal
$W$-algebra, this is precisely the Ramond Zhu algebra denoted
$A_{\mathrm R}(V)$ in the present paper; no passage from an ordinary Zhu
algebra to a Ramond Zhu algebra is being made implicitly.

\begin{proposition}[Finite BRST realization of $B_\ell$]
\label{prop:B-finite-BRST}
\begin{equation}
   B_\ell
   \cong
   H^0_{f_\theta}
   \!\left(U(\mathfrak g)/I_\ell^{\mathrm{cand}}\right).
\end{equation}
\end{proposition}

\begin{proof}
Apply Arakawa's $L_0$-twisted Zhu functor to
\eqref{eq:exact-DS-Q}.  The affine algebra $Q_\ell$ is integrally graded,
so its $L_0$-twisted Zhu algebra is the usual affine Zhu algebra
$A(Q_\ell)\cong U(\mathfrak g)/I_\ell^{\mathrm{cand}}$, whereas on
the reduced half-integrally graded $W$-algebra the same functor is the
Ramond Zhu functor by the convention just recalled.  Theorem~8.5 of
\cite{Arakawa2015} therefore gives exactly the displayed isomorphism.
\end{proof}

\begin{corollary}[The quotient map from the finite $W$-algebra]
\label{cor:H-to-B-quotient}
There is a canonical surjective algebra homomorphism
\begin{equation}
\label{eq:H-to-B-quotient}
   \pi_\ell:
   H=U(\mathfrak g,f_\theta)
   \twoheadrightarrow B_\ell
\end{equation}
whose kernel is $H^0_{f_\theta}(I_\ell^{\mathrm{cand}})$.  Equivalently,
\begin{equation}
\label{eq:H-to-B-exact}
   0\longrightarrow H^0_{f_\theta}(I_\ell^{\mathrm{cand}})
   \longrightarrow H
   \xrightarrow{\,\pi_\ell\,} B_\ell
   \longrightarrow0.
\end{equation}
Consequently every $B_\ell$-module is canonically an $H$-module by
restriction of scalars along $\pi_\ell$.
\end{corollary}

\begin{proof}
For an ideal $N\subset V^k(\mathfrak g)$, let $J_N\subset U(\mathfrak g)$
be its Zhu ideal.  In the proof of \cite[Theorem~8.5]{Arakawa2015},
exactness of finite BRST reduction gives
$$
   0\longrightarrow H^0_{f_\theta}(J_N)
   \longrightarrow U(\mathfrak g,f_\theta)
   \longrightarrow H^0_{f_\theta}(U(\mathfrak g)/J_N)
   \longrightarrow0.
$$
Take $N=\langle s_A,s_D\rangle$, so that
$J_N=I_\ell^{\mathrm{cand}}$, and identify the final term with
$B_\ell$ by Proposition~\ref{prop:B-finite-BRST}.
\end{proof}

We also need the following filtered-Zhu consequence for the exhaustion argument.
Equip $U(\mathfrak g)$ with the PBW filtration and $A(Q_\ell)$ with
Zhu's filtration.  Under the canonical affine identification
$$
   A(Q_\ell)\cong U(\mathfrak g)/I_\ell^{\mathrm{cand}},
$$
these are the same quotient filtration: the degree-one class
$x(-1)\mathbf1$ maps to $x\in\mathfrak g$.  Thus the standard Poisson
surjection
$R_{Q_\ell}\twoheadrightarrow\operatorname{gr}A(Q_\ell)$
\cite[Corollary~5.3]{ArakawaMoreau2018} is a surjection onto
$\operatorname{gr}_{\mathrm{PBW}}(U(\mathfrak g)/I_\ell^{\mathrm{cand}})$.

\begin{lemma}[Filtered-Zhu containment]
\label{lem:filtered-Zhu-containment}
\begin{equation}
   \operatorname{Var}_{\mathrm{PBW}}
   \!\left(U(\mathfrak g)/I_\ell^{\mathrm{cand}}\right)
   \subset\overline{\mathcal O_{\min}}.
\end{equation}
Consequently, if $I_\ell^{\mathrm{cand}}\subset J$ and $J$ is
primitive, then
$$
   \mathcal V(J)=\{0\}
   \quad\text{or}\quad
   \mathcal V(J)=\overline{\mathcal O_{\min}}.
$$
\end{lemma}

\begin{proof}
Taking spectra of the preceding Poisson surjection gives the first
containment because
$X_{Q_\ell}=\overline{\mathcal O_{\min}}$.  If
$I_\ell^{\mathrm{cand}}\subset J$, the quotient map gives
$\mathcal V(J)\subset
\operatorname{Var}_{\mathrm{PBW}}(U(\mathfrak g)/I_\ell^{\mathrm{cand}})$.
By Joseph's irreducibility theorem, the associated variety of a primitive
ideal is the closure of a single nilpotent orbit
\cite[Chapter~9]{Joseph1995}.  Since $\mathcal O_{\min}$ is the unique
minimal nonzero nilpotent orbit, only the two displayed possibilities
remain.
\end{proof}
\section{Finite W-algebras and the membership theorem}
\label{sec:membership}

Retain the preceding notation, and let
$$
   H:=U(\mathfrak g,f_\theta)
$$
be the finite $W$-algebra associated with the minimal nilpotent.
By Proposition~\ref{prop:B-finite-BRST},
\begin{equation}
\label{eq:membership-B-finite-BRST}
   B_\ell
   \cong
   H^0_{f_\theta}
   \!\left(
      U(\mathfrak g)/I_\ell^{\mathrm{cand}}
   \right).
\end{equation}
We construct a simple $B_\ell$-module whose Skryabin lift has
annihilator $J_{2,\ell}$.  By
Corollary~\ref{cor:single-membership-test}, this is enough to prove descent.

\subsection{A double Ramond spectral flow}

Recall that
$$
   \mathfrak g^\natural
   =
   \mathfrak s\oplus\mathfrak d,
   \qquad
   \mathfrak s\cong A_1,
   \qquad
   \mathfrak d\cong D_r,
$$
with current levels
$$
   k_A^\natural=K=r-1,
   \qquad
   k_D^\natural=1.
$$
We identify weights and coweights using the simply-laced normalization and
write
$$
   \varpi^\vee=\varpi=\frac{\theta_A}{2},
   \qquad
   \eta_1^\vee=\eta_1.
$$
Instead of the standard $A_1$-Ramond coweight $\varpi^\vee$, consider
\begin{equation}
   x_*:=\varpi^\vee+\eta_1^\vee.
\end{equation}
Let
$$
   H_*:=J^{\{x_*\}}
   \in
   \widetilde{\mathcal W}_\ell{}_1.
$$
We apply Li's delta-operator construction to the adjoint module of
$\widetilde{\mathcal W}_\ell$; see \cite{Li1996}.  Denote the resulting
twisted module by
$$
   \widetilde{\mathcal W}_\ell^{\,R,*}.
$$

\begin{lemma}
The inner automorphism
$$
   \exp(2\pi i\,x_{*,0})
$$
is the Ramond automorphism of the minimal $W$-algebra.  Equivalently, it
agrees with the automorphism obtained from the single coweight
$\varpi^\vee$.
\end{lemma}

\begin{proof}
For every root $\gamma$ of $\mathfrak g^\natural$,
$$
   \gamma(\varpi^\vee)\in\mathbb Z,
   \qquad
   \gamma(\eta_1^\vee)\in\mathbb Z.
$$
Hence both current factors are fixed by the inner automorphism.

The weight-$\frac32$ generators are parametrized by
$$
   U=\mathfrak g_{-\frac12}
   \cong
   L_{A_1}(\varpi)\boxtimes L_{D_r}(\eta_1).
$$
The two $A_1$-weights are $\pm\varpi$, and
$$
   (\pm\varpi)(\varpi^\vee)\equiv\frac12\pmod{\mathbb Z}.
$$
Every weight of the $D_r$-vector representation belongs to
$\{\pm\varepsilon_j\}$, and therefore has integral pairing with
$\eta_1^\vee$.  Consequently every weight of $U$ has
$$
   \nu(x_*)\equiv\frac12\pmod{\mathbb Z}.
$$
Thus the inner automorphism acts by $-1$ on every
$G^{\{u\}}$ and fixes the current fields and the conformal vector.  This is
precisely the Ramond automorphism.
\end{proof}

For a weight-one current $J^{\{x\}}$, Li twisting gives
$$
   L_0^{R}
   =
   L_0+x_0+\frac{\kappa_x}{2},
$$
where $\kappa_x\mathbf1=J^{\{x\}}_{(1)}J^{\{x\}}$.  Since the two current
factors are orthogonal,
$$
\begin{aligned}
   \kappa_{x_*}
   &=
   K(\varpi^\vee\mid\varpi^\vee)
   +(\eta_1^\vee\mid\eta_1^\vee)\\
   &=
   \frac{r-1}{2}+1
   =
   \frac{r+1}{2}.
\end{aligned}
$$
Hence
\begin{equation}
\label{eq:double-vacuum-energy}
   h_*=\frac{\kappa_{x_*}}2=\frac{r+1}{4}.
\end{equation}

\begin{lemma}[Lower boundedness]
The twisted adjoint module
$\widetilde{\mathcal W}_\ell^{\,R,*}$ is lower bounded, and
$$
   \operatorname{Spec}L_0^{R,*}
   \subset
   h_*+\mathbb Z_{\ge0}.
$$
In particular, the vacuum vector $\mathbf1$ belongs to the lowest Ramond
eigenspace.
\end{lemma}

\begin{proof}
The minimal $W$-algebra is strongly generated by the weight-one currents,
the fields $G^{\{u\}}$ of conformal weight $3/2$, and the conformal
vector.  The $x_*$-charges of current root vectors belong to
$$
   \{-1,0,1\}.
$$
The $x_*$-charges of the weights of $U$ belong to
$$
   \left\{
      -\frac32,-\frac12,\frac12,\frac32
   \right\}.
$$
Thus every current generator has twisted excess
$$
   1+q\ge0,
$$
and every $G$-generator has twisted excess
$$
   \frac32+q\ge0.
$$
To make the passage from strong generators to arbitrary composite fields
explicit, for a homogeneous $x_{*,0}$-eigenvector $a$ put
$$
   \epsilon_*(a):=\operatorname{wt}(a)+q_*(a).
$$
Since $x_{*,0}$ is a derivation of the vertex algebra, charges add under
products.  Thus for homogeneous $a,b$ and $n\ge0$,
\begin{equation}
\label{eq:twisted-excess-additivity}
\begin{aligned}
   \operatorname{wt}\!\left(a_{(-n-1)}b\right)
      &=\operatorname{wt}(a)+\operatorname{wt}(b)+n,\\
   q_*\!\left(a_{(-n-1)}b\right)
      &=q_*(a)+q_*(b),
\end{aligned}
\end{equation}
and hence
$$
   \epsilon_*\!\left(a_{(-n-1)}b\right)
   =\epsilon_*(a)+\epsilon_*(b)+n\ge0.
$$
The conformal vector has excess $2$, and derivatives increase excess by
positive integers.  Strong generation therefore implies that every state,
including all normally ordered products and their derivatives, has
nonnegative integral excess above the vacuum energy $h_*$.  In particular,
no contraction term in a composite field can create a lower twisted
$L_0$-eigenvalue.
\end{proof}

\begin{remark}
Some current modes and some $G$-modes have zero twisted excess.  We
therefore do \emph{not} claim that the complete lowest Ramond eigenspace is
an irreducible $\mathfrak g^\natural$-module, or that it is generated only
by the vacuum under the $A_1$-current algebra.  The proof below uses only
the vacuum line inside that lowest eigenspace.
\end{remark}

The shifted Cartan zero modes are
$$
   J^{\{h\},R,*}_0
   =
   J^{\{h\}}_0+
   K(\varpi^\vee\mid h_A)
   +(\eta_1^\vee\mid h_D),
$$
for $h=h_A+h_D\in\mathfrak h^\natural$.  Hence the vacuum has shifted
$\mathfrak g^\natural$-weight
\begin{equation}
\label{eq:double-vacuum-weight}
   \lambda_*
   =
   K\varpi+\eta_1
   =
   (r-1)\varpi+\eta_1.
\end{equation}

\subsection{Premet parameters of the twisted vacuum}
\label{subsec:Premet-vacuum-parameters}

We compare \eqref{eq:double-vacuum-weight} with Premet's highest-weight
coordinates for the minimal finite $W$-algebra \cite{Premet2007}.  Choose
$e_\theta\in\mathfrak g_\theta$ so that
$(e_\theta,\theta^\vee,f_\theta)$ is the standard minimal
$\mathfrak{sl}_2$-triple.  The invariant form is normalized as in the
Ramond presentation: with
$$
   x=\frac12\theta^\vee
$$
one has $(x\mid x)=\frac12$.  Invariance then gives
$(e_\theta\mid f_\theta)=1$, which is precisely Premet's normalization.
Choose $w\in W(D_\ell)$ with
$$
   w(-\theta)=\beta:=\alpha_2,
$$
choose a representative $\dot w\in N_G(\mathfrak h)$, and write $w$
also for the Lie-algebra automorphism $\operatorname{Ad}(\dot w)$.
Transport the standard triple by setting
\begin{equation}
\label{eq:Premet-transported-triple}
   e_{\mathrm P}:=w(f_\theta),\qquad
   h_{\mathrm P}:=-w(\theta^\vee),\qquad
   f_{\mathrm P}:=w(e_\theta).
\end{equation}
After the compatible choice of root vectors this is Premet's triple with
$e_{\mathrm P}\in\mathfrak g_\beta$,
$h_{\mathrm P}=\beta^\vee$,
$f_{\mathrm P}\in\mathfrak g_{-\beta}$, and
$(e_{\mathrm P}\mid f_{\mathrm P})=1$.  We choose the corresponding
positive system; its reductive centralizer is
$$
   A_1\oplus D_r.
$$
If $U=\mathfrak g_{-1/2}$ is defined using the original minimal grading,
then
\begin{equation}
\label{eq:U-to-Premet-degree-one}
   w(U)=\mathfrak g_{\mathrm P}(1)=\mathfrak z_\chi(1),
\end{equation}
because $[\theta^\vee,u]=-u$ for $u\in U$, whereas
$[h_{\mathrm P},w(u)]=w(u)$.  Equivalently, the Ramond convention grades
by $x=\theta^\vee/2$, while Premet uses the standard Cartan element
$h_{\mathrm P}=2x_{\mathrm P}$.  Thus the Ramond degree $-\tfrac12$
space is transported to Premet degree $+1$; this accounts explicitly for
the factor of two in the two grading conventions.

Premet's highest-weight theory involves a shift
$\bar\delta\in(\mathfrak h^\natural)^*$.  The elementary type-$D$ root
calculation is recorded in Appendix~\ref{app:type-D-calculations}.

\begin{lemma}[Premet shift]
\label{lem:Premet-shift-D}
In the $\beta=\alpha_2$ model,
\begin{equation}
   \bar\delta
   =
   (r-1)\varpi+\eta_1.
\end{equation}
Consequently,
$$
   \lambda_*=\bar\delta.
$$
\end{lemma}

Let $C\in Z(H)$ be Premet's quadratic central generator, normalized so that
the finite $W$-module associated with a Harish--Chandra parameter
$\mu$ has eigenvalue $(\mu,\mu+2\rho)$.  The type-$D_\ell$ entry in
\cite[Theorem~6.1(iv)]{Premet2007} gives
$$
   c_0=-\ell(\ell-2)=-r(r+2).
$$
We now record carefully the normalization linking this generator to the
Ramond conformal class.  Put
$$
   L'=2(k+h^\vee)L+\frac12p_{\mathfrak g}(k).
$$
Equation~(5.1) of \cite{KacMosenederFrajriaPapi2025} gives the Ramond Zhu
relation for the generators attached to $U=\mathfrak g_{-1/2}$.  Although
the main setup of that section is written for basic Lie superalgebras,
\cite[Remark~5.8]{KacMosenederFrajriaPapi2025} states explicitly that for an
ordinary simple Lie algebra the Ramond Zhu algebra is still defined by those
same commutation relations and is isomorphic to Premet's minimal finite
$W$-algebra.  Write $Q_{\rm R}(u,v)$ for the noncentral quadratic-current
term in \cite[Equation~(5.1)]{KacMosenederFrajriaPapi2025}, so that
\begin{equation}
\label{eq:Ramond-quadratic-relation}
 [\mathsf G(u),\mathsf G(v)]
 =\langle u,v\rangle_{\rm R}
     \bigl(\Omega^\natural-L'\bigr)+Q_{\rm R}(u,v).
\end{equation}
There is one normalization scalar which must not be suppressed.  Let
$e_{\rm R}\in\mathfrak g_\theta$ denote the element called $e$ in
\cite[Section~3]{KacMosenederFrajriaPapi2025}.  Their convention is
$[e_{\rm R},f_{\rm R}]=x$, rather than the standard
$[e_\theta,f_\theta]=\theta^\vee=2x$.  Since
$\mathfrak g_\theta$ is one-dimensional, write
$$
   e_{\rm R}=\tau e_\theta,
   \qquad \tau\in\mathbb C^\times.
$$
Equation~(3.4) of that paper reads
\begin{equation}
\label{eq:Ramond-pairing-orientation}
   \langle u,v\rangle_{\mathrm R}
   =(e_{\rm R}\mid[u,v]).
\end{equation}
On the Premet side, by \eqref{eq:Premet-transported-triple}, invariance of the
form, and $\widetilde u=w(u)$, $\widetilde v=w(v)$,
$$
   (f_{\mathrm P}\mid[w(u),w(v)])
   =(e_\theta\mid[u,v]).
$$
Consequently, if
\begin{equation}
\label{eq:relative-pairing-scale}
   \kappa:=\tau^{-1},
\end{equation}
then the exact relation between the two pairings is
\begin{equation}
\label{eq:Ramond-Premet-pairing-scale}
   (f_{\mathrm P}\mid[w(u),w(v)])
   =\kappa\,\langle u,v\rangle_{\rm R}.
\end{equation}
Keeping $\kappa$ explicit is essential: its value depends on the relative
choice of root-vector normalization, whereas the central normalization proved
below does not.

The comparison used here is the canonical filtered Zhu--finite-$W$
comparison, not an arbitrary algebra isomorphism.  De Sole--Kac
\cite[Theorems~4.20 and~5.10]{DeSoleKac2006} identify the
$L_0$-twisted Zhu algebra of the affine $W$-algebra canonically with the
finite $W$-algebra; the map is induced on BRST cohomology by the Zhu
projection and is compatible with the standard filtrations.  Together with
\cite[Remark~5.8]{KacMosenederFrajriaPapi2025}, we use the resulting standard
PBW identification $\iota$ for which the degree-zero
$\mathfrak g^\natural$-generators agree.  Choose dual bases
$\{a_\alpha\}$, $\{a^\alpha\}$ of $\mathfrak g^\natural$ for the
invariant form and set
$$
   \Omega^\natural:=\sum_\alpha a^\alpha*a_\alpha,
   \qquad
   \Theta_{\rm Cas}:=
      \sum_\alpha\Theta(a^\alpha)\Theta(a_\alpha),
$$
where $*$ is the Ramond Zhu product.  Since the canonical map preserves
this product and the degree-zero generators, one has the exact equality
\begin{equation}
\label{eq:Casimir-PBW-identification}
   \iota(\Omega^\natural)=\Theta_{\rm Cas}.
\end{equation}
On the associated graded of the irreducible degree-one space
\eqref{eq:U-to-Premet-degree-one}, equivariance first gives a scalar
$s\ne0$.  In the present type-D case there can be no lower-PBW correction.
Indeed,
$$
   U=L_{A_1}(\varpi)\boxtimes L_{D_r}(\eta_1)
$$
belongs to the nontrivial $(\varpi,\eta_1)$-coset of
$P(A_1\oplus D_r)/Q(A_1\oplus D_r)$, whereas every weight occurring in the
adjoint $\mathfrak g^\natural$-module generated by lower PBW monomials lies
in the root lattice.  Hence there is no
$\mathfrak g^\natural$-equivariant lower-filtration map from $U$ to the
lower PBW part.  The generator comparison is therefore the exact equality
\begin{equation}
\label{eq:degree-one-generator-scale}
   \iota(\mathsf G(u))=s\,\Theta(w(u)).
\end{equation}

We now fix the central normalization without assuming any coefficient
identity for the quadratic-current terms.  This is the point at which the
multiplicative normalization of the quadratic Casimir and the additive
constant must be separated.

\begin{lemma}[Exact Ramond--Premet central normalization]
\label{lem:exact-Ramond-Premet-normalization}
For the canonical filtered identification $\iota$,
\begin{equation}
\label{eq:Ramond-Premet-normalization}
   \iota\!\left(
      2(k+h^\vee)L+\frac12p_{\mathfrak g}(k)
   \right)
   =C-c_0.
\end{equation}
Moreover the coefficient comparison determines the relative normalizations
exactly:
\begin{equation}
\label{eq:derived-pairing-coefficient-comparison}
   \langle u,v\rangle_{\rm R}
   =-\frac{s^2}{2}
      (f_{\mathrm P}\mid[w(u),w(v)]),
\end{equation}
and the noncentral quadratic terms satisfy
\begin{equation}
\label{eq:derived-quadratic-term-comparison}
   \iota\bigl(Q_{\rm R}(u,v)\bigr)
   =\frac{s^2}{2}Q_{\rm P}(w(u),w(v)).
\end{equation}
The numerical value of $s$ itself depends on the relative root-vector
normalization encoded by $\kappa$ and is not needed separately.
\end{lemma}

\begin{proof}
Put
$$
   Z_{\rm R}:=2(k+h^\vee)L+\frac12p_{\mathfrak g}(k).
$$
By \cite[Equation~(5.1)]{KacMosenederFrajriaPapi2025}, $Z_{\rm R}$ is the
central generator through which the level enters the Ramond Zhu
presentation; after replacing $L$ by $Z_{\rm R}$, the presentation is
independent of $k$ for $k\ne-h^\vee$.

We first determine the coefficient of Premet's quadratic central generator.
Throughout this paragraph the filtration on the finite $W$-algebra is
Premet's Kazhdan filtration: the generators coming from
$\mathfrak z_\chi(0)$ have degree $2$, those coming from
$\mathfrak z_\chi(1)$ have degree $3$, and the quadratic central generator
$C$ has degree $4$.  We use the corresponding filtration on the Ramond Zhu
algebra under the canonical De Sole--Kac comparison.

The canonical associative isomorphism is
\cite[Theorem~5.10]{DeSoleKac2006}.  The quasiclassical diagram of
\cite[Section~6, especially the discussion following Example~6.11]{DeSoleKac2006}
identifies the conformal-weight filtration on its finite Zhu algebra with the
Kazhdan filtration and identifies the associated graded with the classical
finite $W$-algebra, i.e. the Poisson algebra of the Slodowy slice.  Thus the
canonical Zhu--finite-$W$ comparison respects the filtration used here and
its principal-symbol map is the classical Hamiltonian-reduction map.  In the BRST
complex the energy-momentum element is the sum of the affine Sugawara term,
the improvement term, and the neutral and charged ghost terms; see
\cite[Section~5.1]{DeSoleKac2006}.  The Sugawara normalization
\cite[Equation~(1.60)]{DeSoleKac2006} is
$$
   2(k+h^\vee)L^{\mathfrak g}
   =\sum_i :x^i x_i:
$$
for dual bases $\{x_i\}$, $\{x^i\}$ of $\mathfrak g$.  Under the
classical BRST reduction, the improvement and ghost summands give the standard
BRST completion of this invariant class, while the scalar
$p_{\mathfrak g}(k)/2$ does not change its principal Kazhdan symbol.
Consequently the degree-four principal symbol of $Z_{\rm R}$ is the
Hamiltonian reduction of the normalized invariant quadratic polynomial
$$
   q=\sum_i x^i x_i
$$
with coefficient one.  Premet's element $C$ is the image of the same
quadratic Casimir under the canonical center map
$Z(U(\mathfrak g))\xrightarrow{\sim}Z(H)$; see
\cite[Corollary~5.1]{Premet2007}.  Thus
\begin{equation}
\label{eq:principal-symbol-ZR-C}
   \operatorname{gr}_4\iota(Z_{\rm R})=\operatorname{gr}_4 C.
\end{equation}

We also need to exclude a hidden lower-degree central correction.  The center
isomorphism is strict for the Kazhdan filtration: in the Harish--Chandra
realization the linear coordinates on $\mathfrak h^*$ have Kazhdan degree
$2$, and the induced isomorphism
$
 \mathbb C[\mathfrak h^*]^{W\!\cdot}\xrightarrow{\sim}Z(H)
$
is strict; see
\cite[Section~3.4, especially Lemma~3.3 and the paragraph following
Equation~(3.7)]{AmbrosioCarnovaleEspositoTopley2024}.  Since
$\mathfrak g$ is simple, the smallest positive ordinary degree of a
Weyl-invariant polynomial is $2$.  Hence the smallest positive Kazhdan
degree in $Z(H)$ is $4$, and therefore
$$
   Z(H)\cap F_3H=\mathbb C\,1.
$$
By \eqref{eq:principal-symbol-ZR-C}, the central element
$\iota(Z_{\rm R})-C$ lies in $F_3H$.  It must therefore be a scalar.  Thus
there is a unique $b\in\mathbb C$ such that
\begin{equation}
\label{eq:central-normalization-up-to-scalar}
   \iota(Z_{\rm R})=C+b.
\end{equation}
This fixes the multiplicative coefficient of $C$ independently of the
normalization of the degree-three generators and leaves only the additive
scalar $b$ to determine.

It remains to determine $b$.  Let $Q_{\rm P}(w(u),w(v))$ denote the
quadratic-current sum in Premet's presentation.  Theorem~6.1(iv) of
\cite{Premet2007} gives
\begin{equation}
\label{eq:Premet-quadratic-relation}
 [\Theta(w(u)),\Theta(w(v))]
 =\frac12(f_{\rm P}\mid[w(u),w(v)])
      \bigl(C-\Theta_{\rm Cas}-c_0\bigr)
   +\frac12Q_{\rm P}(w(u),w(v)).
\end{equation}
Apply $\iota$ to \eqref{eq:Ramond-quadratic-relation}.  Using
\eqref{eq:Ramond-Premet-pairing-scale},
\eqref{eq:Casimir-PBW-identification},
\eqref{eq:degree-one-generator-scale}, and
\eqref{eq:central-normalization-up-to-scalar}, we obtain
\begin{equation}
\label{eq:coefficient-comparison-before-normalization}
\begin{aligned}
 s^2[\Theta(w(u)),\Theta(w(v))]
   &=\kappa^{-1}(f_{\rm P}\mid[w(u),w(v)])
       \bigl(\Theta_{\rm Cas}-C-b\bigr)\\
   &\qquad+\iota\bigl(Q_{\rm R}(u,v)\bigr).
\end{aligned}
\end{equation}
No identification of the two $Q$-terms, and no choice of the relative
root-vector scale $\kappa$, has been used here.

We now justify the coefficient comparison in the PBW basis without hiding
any scalar contribution in the quadratic terms.  Equation~(5.1) of
\cite{KacMosenederFrajriaPapi2025} writes $Q_{\rm R}(u,v)$ explicitly as a
sum of symmetrized products of degree-zero
$\mathfrak g^\natural$-current generators.  Likewise,
\cite[Theorem~6.1(iv)]{Premet2007} writes
$Q_{\rm P}(w(u),w(v))$ explicitly as a sum of symmetrized products of the
corresponding $\Theta(\mathfrak g^\natural)$-generators.  Since $\iota$
identifies the current generators exactly and preserves multiplication,
$\iota(Q_{\rm R}(u,v))$ has the same structural form.  Reordering such
quadratic current monomials into Premet's ordered PBW basis can only use
$$
   [\Theta(x),\Theta(y)]=\Theta([x,y])
   \qquad(x,y\in\mathfrak g^\natural),
$$
so it can create at most linear current terms.  It cannot create either the
independent degree-four generator $C$ or the unit $1$.  Consequently the
PBW coefficients of both $C$ and $1$ in
$Q_{\rm P}(w(u),w(v))$ and $\iota(Q_{\rm R}(u,v))$ are zero.

Choose $u,v$ with
$(f_{\rm P}\mid[w(u),w(v)])\ne0$, which is possible because this alternating
form on $\mathfrak z_\chi(1)$ is nondegenerate.  Substitute
\eqref{eq:Premet-quadratic-relation} into the left-hand side of
\eqref{eq:coefficient-comparison-before-normalization} and compare first the
coefficient of $C$.  This gives
\begin{equation}
\label{eq:s-kappa-coefficient}
   \frac{s^2}{2}=-\kappa^{-1}.
\end{equation}
Combining this with \eqref{eq:Ramond-Premet-pairing-scale} yields immediately
\eqref{eq:derived-pairing-coefficient-comparison}.  Comparing next the scalar
coefficient gives
$$
   -\frac{s^2}{2}c_0=-\kappa^{-1}b.
$$
Using \eqref{eq:s-kappa-coefficient}, we obtain $b=-c_0$, proving
\eqref{eq:Ramond-Premet-normalization}.  Finally, substitute
$b=-c_0$ and \eqref{eq:s-kappa-coefficient} back into
\eqref{eq:coefficient-comparison-before-normalization}.  The complete
central-current part cancels and one is left with
$$
   \iota\bigl(Q_{\rm R}(u,v)\bigr)
   =\frac{s^2}{2}Q_{\rm P}(w(u),w(v)),
$$
which proves \eqref{eq:derived-quadratic-term-comparison}.

Thus both exact coefficient identities are conclusions of the filtered and
PBW comparison, with the relative root-vector scale tracked explicitly rather
than silently fixed.  In particular, the value of $b$, and hence the
identity $\iota(Z_{\rm R})=C-c_0$, is independent of $\kappa$.  As an
independent check,
\cite[Remark~5.8]{KacMosenederFrajriaPapi2025} obtains from the same
$k$-independent Ramond presentation the constant
$$
   c_0=-\frac12p_{\mathfrak g}(-h^\vee),
$$
in agreement with the normalization above.
\end{proof}

We henceforth suppress $\iota$ and regard
\eqref{eq:Ramond-Premet-normalization} as the equality
$L'=C-c_0$ in the fixed PBW identification.  There is no residual
multiplicative or additive ambiguity: the coefficient of $C$ is fixed by
the quadratic associated-graded symbol, and the scalar shift is fixed by the
PBW coefficient comparison in the lemma.
For $D_\ell$, \cite[Section~3.1.2]{Jin2026} gives
$p_{D_\ell}(k)=(k+2)(k+r)$.  Hence at $k=-1$, since
$k+h^\vee=2r+1=N$,
\begin{equation}
\label{eq:Premet-C-L-relation}
   C-c_0
   =
   2N L+\frac{r-1}{2}.
\end{equation}

\begin{lemma}
\label{lem:vacuum-C-zero}
The quadratic central element $C$ acts on the twisted vacuum by zero.
\end{lemma}

\begin{proof}
By \eqref{eq:double-vacuum-energy} and
\eqref{eq:Premet-C-L-relation},
$$
\begin{aligned}
   C\mathbf1
   &=
   \left(
      -r(r+2)
      +2(2r+1)\frac{r+1}{4}
      +\frac{r-1}{2}
   \right)\mathbf1\\
   &=0.
\end{aligned}
$$
\end{proof}

Let
$$
   T_*:=
   \widetilde{\mathcal W}_\ell^{\,R,*}[h_*]
$$
be the lowest Ramond eigenspace.  It is naturally a module over $B_\ell$, hence, by restriction of scalars
along the explicit quotient map $\pi_\ell:H\twoheadrightarrow B_\ell$ from
Corollary~\ref{cor:H-to-B-quotient}, a module over $H$.

\begin{lemma}[Triangular generators and their Li zero modes]
\label{lem:Premet-positive-Li-modes}
Use the positive system of Appendix~\ref{app:type-D-calculations}.  Under the
PBW-normalized Ramond-Zhu/finite-$W$ identification, the generators
$\Theta(e_\alpha)$, for
$e_\alpha\in\mathfrak n^+(0)$, are nonzero scalar multiples of the
Ramond Zhu classes of $J^{\{e_\alpha\}}$.  Likewise
$\Theta(u_i^*)$, for $u_i^*=[e_\beta,z_i^*]\in\mathfrak n^+(1)$,
are nonzero scalar multiples of the classes of the corresponding
$G^{\{v_i\}}$.  If $a$ is one of these current or $G$-generators and
has $x_{*,0}$-charge $q$, then, in untwisted adjoint-mode notation,
\begin{equation}
\label{eq:Li-zero-mode-shift}
   a^{R,*}_0=a_q.
\end{equation}
\end{lemma}

\begin{proof}
   The current statement follows from the degree-zero part of the canonical
   De Sole--Kac PBW comparison.  For the degree-one statement, the
   associated-graded comparison and the root-lattice obstruction in the
   paragraph preceding Lemma~\ref{lem:exact-Ramond-Premet-normalization} show
   that no lower-PBW correction is possible.  Thus the generator-level
   comparison is exact;
   see also
   \cite[Equation~(5.1) and Remark~5.8]{KacMosenederFrajriaPapi2025}.
   The nonzero normalizing scalars do not affect a highest-vector condition.

Li's delta operator gives \eqref{eq:Li-zero-mode-shift} provided
$J^{\{x_*\}}_{(n)}a=0$ for every $n\geq1$.  For
$a=J^{\{e_\alpha\}}$, the current OPE has no pole of order greater than
two, and its double-pole coefficient is proportional to
$(x_*,e_\alpha)=0$; hence all these positive products vanish.  For
$a=G^{\{v_i\}}$, the current--$G$ OPE has only the simple pole encoding
the $\mathfrak g^\natural$-action, so the same vanishing holds.  Thus the
positive-mode exponential in Li's delta operator acts trivially on precisely
the generators used here, leaving only the factor $z^q$, which proves
\eqref{eq:Li-zero-mode-shift}.  Notice that this assertion is not made for
an arbitrary charged vector; in particular, the Cartan-current scalar shift
was computed separately above.
\end{proof}

\begin{proposition}[The vacuum finite-$W$ highest vector]
\label{prop:vacuum-finite-W-highest}
The vacuum vector
$$
   \mathbf1\in T_*
$$
is a finite-$W$ highest-weight vector with Premet highest pair
\begin{equation}
   (\bar\delta,0).
\end{equation}
\end{proposition}

\begin{proof}
The two highest parameters are $\bar\delta$ and $0$ by
\eqref{eq:double-vacuum-weight}, Lemma~\ref{lem:Premet-shift-D}, and
Lemma~\ref{lem:vacuum-C-zero}.  Premet's positive part is generated by
$\Theta(\mathfrak n^+(0))$ and $\Theta(\mathfrak n^+(1))$
\cite[Section~7]{Premet2007}.  Lemma~\ref{lem:Premet-positive-Li-modes}
identifies them, generator by generator, with the positive current and
weight-$3/2$ Ramond zero modes.

For $\Theta(\mathfrak n^+(0))$, Appendix~\ref{app:type-D-calculations}
gives charges $q\in\{0,1\}$.  Premet's $\mathfrak n^+(1)$ is spanned by
$u_i^*=[e_\beta,z_i^*]$, of root $\beta+\gamma_i^*$; since
$\beta(x_*)=0$, its charge is $\gamma_i^*(x_*)$, and the chosen positive
half has charges $q\in\{\frac12,\frac32\}$.  Thus the positive finite-$W$
generators act on $\mathbf1$ through the untwisted modes
$$
   J_q\quad(q=0,1),
   \qquad
   G_q\quad\left(q=\frac12,\frac32\right),
$$
all of which annihilate the vacuum by the creation property.  In particular,
the Li twist does not exchange Premet's chosen positive half with its dual.
Hence $\mathbf1$ is a finite-$W$ highest-weight vector.
\end{proof}

Let
\begin{equation}
   C_*:=B_\ell\mathbf1\subset T_*.
\end{equation}
Since $B_\ell$ is finite dimensional, $C_*$ is finite dimensional.  Any
proper submodule of $C_*$ misses its cyclic generator $\mathbf1$, so a
simple quotient of $C_*$ retains a nonzero image of the vacuum highest
line.  Premet's highest-weight theory therefore gives the following.

\begin{corollary}
\label{cor:Mstar-B-module}
The finite-$W$ simple module
\begin{equation}
   M_*:=L_H(\bar\delta,0)
\end{equation}
is a $B_\ell$-module.
\end{corollary}

\begin{proof}
Take a simple quotient of the cyclic $B_\ell$-module $C_*$.  By
Proposition~\ref{prop:vacuum-finite-W-highest}, it is a simple highest-weight
$H$-module with highest pair $(\bar\delta,0)$.  The uniqueness of the
simple quotient of the corresponding finite-$W$ Verma module
\cite[Section~7]{Premet2007} identifies it with
$L_H(\bar\delta,0)$.
\end{proof}

\begin{lemma}[Rank-one Whittaker parameter comparison]
\label{lem:premet-ms-parameter}
Put $\beta=\alpha_2$, and let $\psi_\beta$ be the Whittaker
character supported on the $\beta$-root space, normalized by
$\psi_\beta(e_\beta)=1$.  For $\nu\in\mathfrak h^*$, write
$$
   \lambda=\bar\nu\in(\mathfrak h^\natural)^*,
   \qquad
   m=\langle\nu,\beta^\vee\rangle,
   \qquad
   c_\beta(\nu)=\frac12m(m+2).
$$
Here $\bar\nu$ is the restriction to
$\mathfrak h^\natural=\beta^\perp$, identified with the orthogonal
projection by Premet's bilinear form.  If $M_{\rm P}(\lambda,c)$ denotes
Premet's parabolically induced module and $M_{\rm MS}(\nu,\psi_\beta)$
the Mili\v{c}i\'c--Soergel standard Whittaker module, choose a representative
$\dot s_\beta$ of $s_\beta$, composed if necessary with a torus element,
so that $\operatorname{Ad}(\dot s_\beta)(f_\beta)=e_\beta$.  For a
$\mathfrak g$-module $M$, write ${}^{\dot s_\beta}M$ for the inner
twist in which $x\in\mathfrak g$ acts as
$\operatorname{Ad}(\dot s_\beta^{-1})(x)$.  Then
\begin{equation}
\label{eq:Premet-MS-standard-comparison}
   {}^{\dot s_\beta}M_{\rm P}\!\left(\bar\nu,c_\beta(\nu)\right)
   \cong
   M_{\rm MS}(\nu,\psi_\beta).
\end{equation}
Moreover Premet's Whittaker functor gives
\begin{equation}
\label{eq:Premet-MS-finiteW-parameters}
   \operatorname{Wh}
   M_{\rm P}\!\left(\bar\nu,c_\beta(\nu)\right)
   \cong
   Z_H\!\left(
      \bar\nu+\bar\delta,
      (\nu,\nu+2\rho)
   \right).
\end{equation}
Both sides depend only on the
$\langle s_\beta\rangle$-dot orbit of $\nu$.
\end{lemma}

\begin{proof}
For the character $\psi_\beta$, the Mili\v{c}i\'c--Soergel support is
$\Delta_{\psi_\beta}=\{\beta\}$.  Hence their parabolic and Levi are
\begin{equation}
\label{eq:MS-rank-one-parabolic}
   \mathfrak p_{\psi_\beta}
   =\mathfrak g_{\psi_\beta}\oplus\mathfrak n_\beta,
   \qquad
   \mathfrak g_{\psi_\beta}
   =\mathfrak h^\natural\oplus\mathfrak s_\beta,
   \qquad
   \mathfrak n_\beta
   =\bigoplus_{\gamma\in\Phi^+\setminus\{\beta\}}\mathfrak g_\gamma,
\end{equation}
with $\mathfrak s_\beta\cong\mathfrak{sl}_2$; see
\cite[Section~2]{MilicicSoergel1997}.  This is precisely the parabolic
$\mathfrak p_\beta$ used by Premet in the construction preceding
\cite[Theorem~1.3]{Premet2007}.

We compare the inducing modules themselves, rather than only their final
central characters.  In Premet's notation the inducing module is
$$
   Y_{\rm P}(\lambda,c)
   =U(\mathfrak p_\beta)/I_\beta(\lambda,c),
$$
where the left ideal $I_\beta(\lambda,c)$ is generated by
\begin{equation}
\label{eq:Premet-inducing-relations}
   f_\beta-1,
   \qquad C_\beta-c,
   \qquad h-\lambda(h)\ \ (h\in\mathfrak h^\natural),
   \qquad e_\gamma\ \ (\gamma\in\Phi^+\setminus\{\beta\}),
\end{equation}
with
$$
   C_\beta=e_\beta f_\beta+f_\beta e_\beta
            +\frac12h_\beta^2;
$$
see the definition of $M(\lambda,c)$ in
\cite[Section~7.3]{Premet2007}.  Mili\v{c}i\'c--Soergel define
$$
   Y_{\psi_\beta}(\xi_{\psi_\beta}(\nu),\psi_\beta)
$$
as the simple Whittaker module for the Levi
$\mathfrak g_{\psi_\beta}$ with relative Harish--Chandra character
$\xi_{\psi_\beta}(\nu)$, and then set
\begin{equation}
\label{eq:MS-standard-definition}
   M_{\rm MS}(\nu,\psi_\beta)
   =U(\mathfrak g)\otimes_{U(\mathfrak p_\beta)}
    Y_{\psi_\beta}(\xi_{\psi_\beta}(\nu),\psi_\beta).
\end{equation}
See \cite[Section~2]{MilicicSoergel1997}.  We now check the cyclic
presentation directly.

On the central $\mathfrak h^\natural$-factor the relative
Harish--Chandra character is $\bar\nu$.  On the
$\mathfrak{sl}_2$-factor, if
$m=\langle\nu,\beta^\vee\rangle$, the Casimir
$C_\beta=e_\beta f_\beta+f_\beta e_\beta+\frac12h_\beta^2$
has Harish--Chandra value
$$
   m+\frac12m^2=\frac12m(m+2)=c_\beta(\nu).
$$
These are exactly the two components of
$\xi_{\psi_\beta}(\nu)$.

With $\dot s_\beta$ fixed as in the statement, $s_\beta$ fixes
$\mathfrak h^\natural=\beta^\perp$ pointwise and preserves
$C_\beta$.  Moreover, because $\beta$ is simple,
$s_\beta(\Phi^+\setminus\{\beta\})
 =\Phi^+\setminus\{\beta\}$.  Hence, in the twisted cyclic module,
$$
   e_\beta\cdot 1
   =\operatorname{Ad}(\dot s_\beta^{-1})(e_\beta)\cdot1
   =f_\beta\cdot1=1,
$$
while every $h\in\mathfrak h^\natural$ acts by
$\bar\nu(h)$, the element $C_\beta$ acts by
$c_\beta(\nu)$, and every root vector in $\mathfrak n_\beta$
annihilates the cyclic vector.  Scalar changes of the other root vectors
  caused by the chosen representative $\dot s_\beta$ are irrelevant to
  these zero relations.  In this rank-one Levi, the relative center is generated
  by $\mathfrak h^\natural$ and $C_\beta$.  It remains to check that the
  resulting cyclic Levi module is already simple, rather than merely having the
  correct simple quotient.

  This can be seen directly in rank one.  Suppress the subscript $\beta$, put
  $h=h_\beta$, and for $c\in\mathbb C$ set
  $$
     q_c(t)=\frac c2-\frac t2-\frac{t^2}{4}.
  $$
  On $\mathbb C[t]$ define
  \begin{equation}
  \label{eq:rank-one-difference-model}
  \begin{aligned}
     h\,p(t)&=t p(t),\\
     e\,p(t)&=p(t-2),\\
     f\,p(t)&=q_c(t)p(t+2).
  \end{aligned}
  \end{equation}
  A direct calculation gives
  $[h,e]=2e$, $[h,f]=-2f$, $[e,f]=h$, and
  $$
     ef+fe+\frac12h^2=c.
  $$
  Thus $1\in\mathbb C[t]$ satisfies $e1=1$ and $C_\beta1=c$.
  Conversely, PBW and these two cyclic relations reduce every vector in
  $$
     U(\mathfrak{sl}_2)/
       \bigl(U(\mathfrak{sl}_2)(e-1)+U(\mathfrak{sl}_2)(C_\beta-c)\bigr)
  $$
  to a polynomial in $h$ applied to the cyclic vector.  The model
  \eqref{eq:rank-one-difference-model} shows that those polynomial vectors are
  linearly independent, so this quotient is $\mathbb C[t]$.

  It is simple.  Indeed, a nonzero submodule is an ideal of $\mathbb C[t]$
  because it is stable under multiplication by $t$.  If it contains a monic
  polynomial $p(t)$ of minimal degree, it also contains $p(t-2)$ by
  \eqref{eq:rank-one-difference-model}.  Unless $p$ is constant, the nonzero
  difference $p(t)-p(t-2)$ has smaller degree, a contradiction.  Hence the
  submodule contains $1$ and is the whole module.

  Taking $c=c_\beta(\nu)$ and tensoring with the one-dimensional
  $\mathfrak h^\natural$-module of weight $\bar\nu$ proves that the displayed
  cyclic relations give the simple Levi Whittaker module
  $Y_{\psi_\beta}(\xi_{\psi_\beta}(\nu),\psi_\beta)$ of
  \cite[Section~2]{MilicicSoergel1997}; no further Levi parameter remains.
Thus the displayed relations give a
$\mathfrak p_\beta$-module isomorphism
$$
   {}^{\dot s_\beta}
   Y_{\rm P}\!\left(\bar\nu,c_\beta(\nu)\right)
   \cong
   Y_{\psi_\beta}\!\left(\xi_{\psi_\beta}(\nu),\psi_\beta\right).
$$
The parabolic $\mathfrak p_\beta$ is stable under this inner
conjugation, so induction to $\mathfrak g$ proves
\eqref{eq:Premet-MS-standard-comparison}.  This also makes explicit that
no comparison of merely abstract central characters is being used.

The second finite-$W$ parameter comes from Premet's Theorem~1.3.  For
$\lambda\in(\mathfrak h^\natural)^*$,
\begin{equation}
\label{eq:Premet-Theorem13-explicit}
   \operatorname{Wh}M_{\rm P}(\lambda,c)
   \cong
   Z_H\!\left(
      \lambda+\bar\delta,
      c+(\lambda,\lambda+2\bar\rho)
   \right),
\end{equation}
where $\bar\rho$ is the orthogonal projection of $\rho$ to
$(\mathfrak h^\natural)^*$; see \cite[Theorem~1.3]{Premet2007}.
Since type $D$ is simply laced and
$\langle\rho,\beta^\vee\rangle=1$, the orthogonal decompositions are
$$
   \nu=\bar\nu+\frac m2\beta,
   \qquad
   \rho=\bar\rho+\frac12\beta.
$$
Consequently
\begin{align}
   (\nu,\nu+2\rho)
   &=
   (\bar\nu,\bar\nu+2\bar\rho)
   +\frac12m(m+2) \\
   &=
   (\bar\nu,\bar\nu+2\bar\rho)+c_\beta(\nu).
\end{align}
Substitution into \eqref{eq:Premet-Theorem13-explicit} proves
\eqref{eq:Premet-MS-finiteW-parameters}.
Finally $s_\beta\mathbin\cdot\nu$ has the same restriction
$\bar\nu$ and replaces $m$ by $-m-2$, leaving
$c_\beta(\nu)$ unchanged.  This is also exactly the
Mili\v{c}i\'c--Soergel isomorphism criterion
\cite[Proposition~2.1]{MilicicSoergel1997}.
\end{proof}

\begin{lemma}[Chen's rank-one annihilator criterion]
\label{lem:chen-rank-one-annihilator}
Let $\nu\in\mathfrak h^*$ be integral and
$W_{\psi_\beta}=\langle s_\beta\rangle$-anti-dominant.  Then
\begin{equation}
\label{eq:chen-rank-one-annihilator}
   \operatorname{Ann}_{U(\mathfrak g)}\mathcal L(\nu,\psi_\beta)
   =
   \operatorname{Ann}_{U(\mathfrak g)}L_{\mathfrak g}(\nu).
\end{equation}
\end{lemma}

\begin{proof}
Chen includes ordinary reductive Lie algebras
$\mathfrak g=\mathfrak g_{\bar0}$ among the Lie superalgebras of type I--0
\cite[Section~2.3, Example~(1)]{Chen2021Annihilator}.  In the purely even
case his standard and simple Whittaker modules are the usual
Mili\v{c}i\'c--Soergel modules; in particular the simple top denoted there by
$L(\nu,\psi_\beta)$ is the module denoted here by
$\mathcal L(\nu,\psi_\beta)$.  In this purely even case Chen's dot
action uses the ordinary Weyl vector $\rho_{\bar0}=\rho$, so it agrees
with the dot action used here.  Chen defines
$W_\zeta$-anti-dominance as anti-dominance for the Levi
$\mathfrak l_\zeta$, and \cite[Theorem~B]{Chen2021Annihilator} states that
for every integral $W_\zeta$-anti-dominant weight,
$$
   \operatorname{Ann}_{U(\mathfrak g)}L(\nu,\zeta)
   =
   \operatorname{Ann}_{U(\mathfrak g)}L_{\mathfrak g}(\nu).
$$
For $\zeta=\psi_\beta$, one has
$W_\zeta=\langle s_\beta\rangle$, so the assertion follows.
\end{proof}

\subsection{Identification of the primitive ideal \texorpdfstring{$J_{2,\ell}$}{J2,l}}

Recall from Section~\ref{subsec:primitive-compression} that
$$
   J_{2,\ell}
   =
   \operatorname{Ann}_{U(\mathfrak g)}
   L_{\mathfrak g}(-\alpha_2).
$$
We identify the Skryabin image of $M_*$ without choosing a
left/right-coset convention.

\begin{proposition}[Identification of $J_{2,\ell}$]
\label{prop:node2-identification}
Under Skryabin equivalence,
\begin{equation}
   \operatorname{Ann}_{U(\mathfrak g)}
   \operatorname{Skr}(M_*)
   =
   J_{2,\ell}.
\end{equation}
\end{proposition}

\begin{proof}
Put $\beta=\alpha_2$.  Premet's Theorem~1.3 gives
$$
   \operatorname{Wh}M_{\rm P}(0,0)
   \cong Z_H(\bar\delta,0),
$$
and Whittaker--Skryabin is an exact equivalence.  Since
$M_*=L_H(\bar\delta,0)$ is the unique simple quotient of this
finite-$W$ Verma module, $\operatorname{Skr}(M_*)$ is the unique
simple quotient of $M_{\rm P}(0,0)$.

Lemma~\ref{lem:premet-ms-parameter}, applied to $\nu=0$, gives
$$
   {}^{\dot s_\beta}M_{\rm P}(0,0)
   \cong M_{\rm MS}(0,\psi_\beta).
$$
Hence the corresponding inner twist
${}^{\dot s_\beta}\operatorname{Skr}(M_*)$ is
$\mathcal L(0,\psi_\beta)$.  Since
$$
   s_\beta\mathbin\cdot0=-\beta,
$$
Proposition~2.1 of \cite{MilicicSoergel1997} gives
$$
   \mathcal L(0,\psi_\beta)
   \cong
   \mathcal L(-\beta,\psi_\beta).
$$
For $\zeta=\psi_\beta$, the support of the Whittaker character is
$\Pi_\zeta=\{\beta\}$, hence
$W_\zeta=\langle s_\beta\rangle$.  The weight $-\beta$ is integral and
$$
   \langle-\beta+\rho,\beta^\vee\rangle=-1,
$$
so it is anti-dominant for the rank-one Levi $\mathfrak l_\zeta$.
Lemma~\ref{lem:chen-rank-one-annihilator} therefore gives
$$
   \operatorname{Ann}_{U(\mathfrak g)}
   \mathcal L(-\beta,\psi_\beta)
   =
   \operatorname{Ann}_{U(\mathfrak g)}L_{\mathfrak g}(-\beta)
   =J_{2,\ell}.
$$
The conjugation above is inner, and every two-sided ideal of
$U(\mathfrak g)$ is invariant under inner automorphisms.  Therefore
$\operatorname{Ann}\operatorname{Skr}(M_*)=J_{2,\ell}$, as claimed.
\end{proof}

\subsection{The membership theorem}

\begin{theorem}[Membership theorem]
\label{thm:membership}
For every $\ell\ge5$,
\begin{equation}
\label{eq:membership}
   I_\ell^{\mathrm{cand}}
   \subset
   J_{2,\ell}.
\end{equation}
\end{theorem}

\begin{proof}
By Corollary~\ref{cor:Mstar-B-module} and
\eqref{eq:membership-B-finite-BRST},
$$
   M_*
   \in
   H^0_{f_\theta}
   \!\left(
      U(\mathfrak g)/I_\ell^{\mathrm{cand}}
   \right)\text{-Mod}.
$$
Apply Arakawa's quotient Zhu--Skryabin equivalence
\cite[Theorem~8.6]{Arakawa2015} to the vertex-algebra ideal
$N=\langle s_A,s_D\rangle\subset V^{-1}(\mathfrak g)$.  Its Zhu ideal is
$J_N=I_\ell^{\mathrm{cand}}$, and Theorem~8.6 identifies
$B_\ell\text{-Mod}$ with the Whittaker subcategory
$\mathcal C^{I_\ell^{\mathrm{cand}}}$, with quasi-inverse the Skryabin
functor.  Since $M_*$ is a $B_\ell$-module,
$\operatorname{Skr}(M_*)$ is therefore annihilated by
$I_\ell^{\mathrm{cand}}$.  Thus
$$
   I_\ell^{\mathrm{cand}}\,
   \operatorname{Skr}(M_*)
   =
   0,
$$
and hence
$$
   I_\ell^{\mathrm{cand}}
   \subset
   \operatorname{Ann}_{U(\mathfrak g)}
   \operatorname{Skr}(M_*).
$$
Proposition~\ref{prop:node2-identification} identifies the annihilator on the
right with $J_{2,\ell}$, proving
\eqref{eq:membership}.
\end{proof}

\begin{corollary}[Simultaneous descent]
For every $1\le i\le\ell$, the affine highest-weight module
$$
   L(-\Lambda_0+\mu_i)
$$
factors through $Q_\ell$.
\end{corollary}

\begin{proof}
By Theorem~\ref{thm:primitive-annihilator-compression},
$$
   \operatorname{Ann}_{U(\mathfrak g)}
   L_{\mathfrak g}(\mu_i)
   =
   J_{2,\ell}.
$$
Theorem~\ref{thm:membership} gives
$$
   I_\ell^{\mathrm{cand}}
   \subset
   J_{2,\ell}.
$$
Now apply Corollary~\ref{cor:single-membership-test}.
\end{proof}

\section{The exhaustion theorem}
\label{sec:exhaustion}

Retain the notation of the preceding sections, in particular the candidate
Zhu ideal $I_\ell^{\mathrm{cand}}\subset U(\mathfrak g)$.

The affine vacuum-block congruence leaves no additional highest weights over
the candidate quotient:
\begin{equation}
\begin{gathered}
   I_\ell^{\mathrm{cand}}
   \subset
   \operatorname{Ann}_{U(\mathfrak g)}L_{\mathfrak g}(\mu),
   \qquad
   \mu+\rho\in W\rho+NQ(D_\ell),\\
   \Longrightarrow\qquad
   \mu\in\{\mu_0,\mu_1,\ldots,\mu_\ell\}.
\end{gathered}
\end{equation}

Write
$$
   J(\mu):=
   \operatorname{Ann}_{U(\mathfrak g)}L_{\mathfrak g}(\mu).
$$
By Lemma~\ref{lem:filtered-Zhu-containment},
$$
   I_\ell^{\mathrm{cand}}\subset J(\mu)
   \quad\Longrightarrow\quad
   \mathcal V(J(\mu))
   =
   \{0\}
   \quad\text{or}\quad
   \overline{\mathcal O_{\min}}.
$$
The two cases are handled separately.

\subsection{The zero-orbit branch}

Assume throughout this subsection that
\begin{equation}
\label{eq:zero-orbit-assumptions}
   I_\ell^{\mathrm{cand}}\subset J(\mu),
   \qquad
   \mathcal V(J(\mu))=\{0\},
   \qquad
   \mu+\rho\in W\rho+NQ(D_\ell).
\end{equation}
Since $J(\mu)$ has zero associated variety,
$L_{\mathfrak g}(\mu)$ is finite dimensional.  Hence
$$
   \mu\in P_+.
$$
Write its dominant orthogonal coordinates as
$$
   \mu=(x_1,\ldots,x_\ell),
   \qquad
   x_1\ge x_2\ge\cdots\ge x_{\ell-1}\ge |x_\ell|.
$$
The affine congruence in \eqref{eq:zero-orbit-assumptions} implies that all
$x_i$ are integers, since both $W\rho$ and $NQ(D_\ell)$ have integral
orthogonal coordinates.

Consider the affine irreducible module
$$
   L(-\Lambda_0+\mu).
$$
Its affine simple-coroot value is
$$
   \left\langle
      -\Lambda_0+\mu,\alpha_0^\vee
   \right\rangle
   =
   -1-\langle\mu,\theta^\vee\rangle<0.
$$
Arakawa's irreducible minimal-reduction theorem therefore gives a nonzero
irreducible $W$-module
\begin{equation}
   \mathcal M_\mu
   :=
   H^0_{\mathrm{DS},f_\theta}
   \!\left(
      L(-\Lambda_0+\mu)
   \right)
   \ne0;
\end{equation}
see \cite[Theorem~6.7.4]{Arakawa2005}.  Since
$I_\ell^{\mathrm{cand}}\subset J(\mu)$, the simple top
$L_{\mathfrak g}(\mu)$ is an $A(Q_\ell)$-module.  The $n=0$
simple-object correspondence of
\cite[Theorem~4.9]{DongLiMason1998}, followed by inflation to the universal
affine vertex algebra as in Corollary~\ref{cor:single-membership-test},
identifies its irreducible Zhu lift with $L(-\Lambda_0+\mu)$.  Hence the
affine module factors through $Q_\ell$, and $\mathcal M_\mu$ factors
through $\widetilde{\mathcal W}_\ell$.

Arakawa's irreducible-reduction theorem identifies the nonzero module
$\mathcal M_\mu$ with the irreducible minimal-$W$ module of the
corresponding highest weight
\cite[Theorem~6.7.4]{Arakawa2005}.  That corresponding highest weight is the
one computed for the reduction of the affine Verma module in
\cite[Theorem~6.3]{KacWakimoto2004}; in particular, its
$\mathfrak g^\natural$-component is the restriction of $\mu$ to
$\mathfrak h^\natural$.  Thus the lowest Neveu--Schwarz space of
$\mathcal M_\mu$ has precisely this $\mathfrak g^\natural$-highest
weight.  We begin with the consequences of the Ramond Zhu relations.

\begin{lemma}[Restricted finite type]
\label{lem:zero-orbit-restricted-type}
Under \eqref{eq:zero-orbit-assumptions},
$$
   \mu
   =
   (x_1,x_2,c,0,\ldots,0),
   \qquad
   c\in\{0,1\}.
$$
If
$$
   a:=x_1-x_2,
$$
then
$$
   0\le a\le r-1.
$$
Equivalently,
$$
   \mu|_{\mathfrak h^\natural}
   =
   a\varpi+c\eta_1.
$$
\end{lemma}

\begin{proof}
The lowest Neveu--Schwarz $\mathfrak g^\natural$-module is finite
dimensional because $L_{\mathfrak g}(\mu)$ is finite dimensional.  Let
$u$ be a vector in this lowest space which is lowest weight for the
$A_1$-factor and highest weight for the $D_r$-factor.

Twist $\mathcal M_\mu$ by the single coweight
$$
   x=\varpi^\vee.
$$
The $A_1$-current charges are in $\{-1,0,1\}$, while the
weight-$\frac32$ generators have $x$-charges $\pm\frac12$.  We use
physical mode indices: current and Virasoro modes have integral indices,
whereas the $G$-modes in the Neveu--Schwarz module have indices in
$\mathbb Z+\frac12$.  If a homogeneous generator $a$ has
$x$-charge $q_x(a)$, then
\begin{equation}
\label{eq:single-twist-mode-increment}
   [L_0^R,a_t]=\bigl(-t+q_x(a)\bigr)a_t;
\end{equation}
the scalar term in $L_0^R=L_0+x_0+\kappa_x/2$ does not affect this
commutator.

Order the modes in a PBW spanning monomial for $\mathcal M_\mu$.  Positive
modes annihilate its lowest Neveu--Schwarz space.  Current zero modes can be
moved to the right and absorbed into the action on that space, while $L_0$
acts there by the lowest conformal weight.  It is therefore enough to
consider negative physical modes applied to vectors in the lowest space.
Formula~\eqref{eq:single-twist-mode-increment} gives the following lower
bounds for their changes of twisted conformal weight:
$$
\begin{array}{c|c|c}
\text{generator} & \text{allowed negative mode} & -t+q_x(a)\\
\hline
J & t=n\le-1 & -n+q_x(J)\ge 1-1=0\\
G & t\in\mathbb Z+\frac12,\ t\le-\frac12
  & -t+q_x(G)\ge \frac12-\frac12=0\\
L & t=n\le-1 & -n\ge1.
\end{array}
$$
Thus every factor in such a monomial has nonnegative twisted-energy
increment.  Inside the lowest Neveu--Schwarz space the untwisted
$L_0$-eigenvalue is constant, while $u$ has the smallest $x_0$-eigenvalue
by its choice as an $A_1$-lowest-weight vector.  No spanning vector can
therefore have twisted conformal weight strictly below that of $u$.  Hence
$u$ lies in the lowest Ramond eigenspace.  Zero increments may occur (for
example for a $G_{-1/2}$-mode of charge $-1/2$), so no irreducibility
statement about the complete Ramond bottom is being used.

The cyclic module
$$
   C_\mu:=B_\ell u
$$
is finite dimensional.  Under the twisted $A_1$-zero modes,
$u$ is a highest-weight vector of highest weight
$$
   (K-a)\varpi.
$$
Finite-dimensionality therefore forces
$$
   K-a\ge0,
$$
or $a\le r-1$.

The $D_r$-factor is unchanged by this twist, and $u$ remains a
$D_r$-highest vector.  The relation $e_D^2=0$, equivalently
Lemma~\ref{lem:allowed-finite-types}, gives
$$
   \left\langle
      \lambda_D,\theta_D^\vee
   \right\rangle
   \le1
$$
for its $D_r$-highest weight $\lambda_D$.  The level-one dominant
weights are
$$
   0,\qquad
   \eta_1,\qquad
   \eta_{r-1},\qquad
   \eta_r
$$
for $r\ge4$, with the analogous four weights for
$D_3\cong A_3$.  The affine congruence makes the orthogonal coordinates of
$\mu$ integral, whereas the two spinor weights have half-integral
orthogonal coordinates.  Thus
$$
   \lambda_D=c\eta_1,
   \qquad
   c\in\{0,1\}.
$$
Dominance of $\mu$ now gives
$$
   \mu=(x_1,x_2,c,0,\ldots,0).
$$
\end{proof}

Set
\begin{equation}
   a:=x_1-x_2,
   \qquad
   s:=x_1+x_2
   =
   \langle\mu,\theta^\vee\rangle.
\end{equation}

\begin{lemma}[Lowest conformal energies]
\label{lem:zero-orbit-energies}
The lowest Neveu--Schwarz conformal weight of $\mathcal M_\mu$ is
\begin{equation}
\label{eq:zero-hW}
   h_W
   =
   \frac{
      s^2+a(a+2)+2c(2r-1)
   }{4N}.
\end{equation}
For the single Li twist $x=\varpi^\vee$, the Ramond ground vector $u$
above has energy
\begin{equation}
\label{eq:zero-hR}
   h_R
   =
   h_W-\frac a2+\frac{r-1}{4}.
\end{equation}
\end{lemma}

\begin{proof}
Kac--Wakimoto's conformal-weight formula
\cite[Equation~(7.6)]{KacWakimoto2004} reads, in the present normalization,
$$
   h_W
   =
   \frac{
      \bigl((\mu+\rho,\theta)-N\bigr)^2-(k+1)^2
      +2(\mu^\natural,\mu^\natural+2\rho^\natural)
   }{4N}.
$$
At $k=-1$,
$$
   (\rho,\theta)=N,
   \qquad
   k+1=0.
$$
Moreover,
$$
   \mu^\natural=a\varpi+c\eta_1,
$$
and, since $c\in\{0,1\}$,
$$
   (a\varpi,a\varpi+2\rho_{A_1})
   =\frac{a(a+2)}2,
   \qquad
   (c\eta_1,c\eta_1+2\rho_{D_r})
   =c(2r-1).
$$
This gives \eqref{eq:zero-hW}.

The vector $u$ has $x$-weight $-a/2$.  For Li's operator
\cite[Proposition~5.4]{Li1996},
$$
   \Delta_x(z)\omega
   =\omega+J^{\{x\}}z^{-1}+\frac{\kappa_x}{2}z^{-2},
   \qquad
   \kappa_x=K(\varpi^\vee\mid\varpi^\vee)=\frac K2.
$$
Hence
$$
   L_0^R=L_0+x_0+\frac{\kappa_x}{2}
   =L_0+x_0+\frac K4,
$$
in agreement with \cite[Equation~(62)]{Jin2026}; this yields
\eqref{eq:zero-hR}.
\end{proof}

The cyclic Ramond ground module $C_\mu=B_\ell u$ may contain several
$\mathfrak g^\natural$-constituents.  We therefore use a uniform Casimir
upper bound rather than attempting to identify the complete bottom space.

\begin{lemma}[Root--vector coset bound]
\label{lem:zero-root-vector-bound}
Every $D_r$-constituent of $C_\mu$ belongs to the root or vector coset of
the $D_r$-weight lattice.  Consequently
\begin{equation}
\label{eq:zero-Cmax}
   \frac{\operatorname{tr}_{C_\mu}\Omega_A}{\dim C_\mu}
   +
   \frac{r+2}{2r}
   \frac{\operatorname{tr}_{C_\mu}\Omega_D}{\dim C_\mu}
   \le
   C_{\max},
\end{equation}
where
\begin{equation}
\label{eq:Cmax-definition}
   C_{\max}
   :=
   \frac{r^2-1}{2}
   +
   \frac{r+2}{2r}(2r-1).
\end{equation}
\end{lemma}

\begin{proof}
The initial $D_r$-highest weight is $c\eta_1$, hence belongs to the
subgroup of $P(D_r)/Q(D_r)$ generated by the vector class.  Current
zero modes change weights by roots, while the non-current generators
$G^{\{u\}}$ carry $D_r$-weights in
$$
   \eta_1+Q(D_r).
$$
Since $\widetilde{\mathcal W}_\ell$ is strongly generated by the currents,
the $G$-fields, and the conformal vector, its Ramond Zhu algebra
$B_\ell$ is generated by their Zhu classes.  Thus every operator used to
generate $C_\mu=B_\ell u$ changes the $D_r$-coset by either the root
class or the vector class (the conformal class has weight zero).
Therefore the cyclic module can meet only the root and vector cosets, never
a spinor coset.  This remains true for $r=3\cong A_3$, where the subgroup
generated by the vector weight $\omega_2$ is
$\{0,\omega_2\}$ modulo the root lattice.

By Lemma~\ref{lem:allowed-finite-types}, the largest $A_1$-Casimir among
allowed constituents is attained at $(r-1)\varpi$ and equals
$$
   \frac{(r-1)(r+1)}2
   =
   \frac{r^2-1}{2}.
$$
Among the allowed root/vector $D_r$-types, the largest Casimir is the
vector value
$$
   c_2(\eta_1)=2r-1.
$$
Averaging over irreducible constituents gives \eqref{eq:zero-Cmax}.
\end{proof}

\begin{proposition}[Zero-orbit energy gap]
Under \eqref{eq:zero-orbit-assumptions},
\begin{equation}
   s<N.
\end{equation}
\end{proposition}

\begin{proof}
Apply the weighted Casimir identity
\eqref{eq:weighted-Casimir} to $C_\mu$:
$$
   Nh_R+\frac{r-1}{4}
   =
   \frac{\operatorname{tr}_{C_\mu}\Omega_A}{\dim C_\mu}
   +
   \frac{r+2}{2r}
   \frac{\operatorname{tr}_{C_\mu}\Omega_D}{\dim C_\mu}.
$$
Combining Lemmas~\ref{lem:zero-orbit-energies} and
\ref{lem:zero-root-vector-bound} gives
\begin{equation}
   s^2
   \le
   4ar-a^2
   +4r+6-\frac4r
   -2c(2r-1).
\end{equation}
Since
$$
   0\le a\le r-1,
   \qquad
   c\ge0,
$$
the right-hand side is at most
$$
   3r^2+2r+5-\frac4r.
$$
But
$$
\begin{aligned}
   N^2-
   \left(
      3r^2+2r+5-\frac4r
   \right)
   &=
   (2r+1)^2-
   \left(
      3r^2+2r+5-\frac4r
   \right)\\
   &=
   r^2+2r-4+\frac4r\\
   &>0
\end{aligned}
$$
for every $r\ge3$.  Since $s\ge0$, this proves $s<N$.
\end{proof}

Now impose the affine congruence.  Recall
$$
   Q(D_\ell)
   =
   \left\{
      (q_1,\ldots,q_\ell)\in\mathbb Z^\ell
      \,\middle|\,
      \sum_iq_i\equiv0\pmod2
   \right\},
$$
and
$$
   W(D_\ell)
$$
is the group of signed permutations with an even number of sign changes.

\begin{proposition}[Zero-orbit rigidity]
\label{prop:zero-orbit-rigidity}
Under \eqref{eq:zero-orbit-assumptions},
$$
   \mu=0.
$$
\end{proposition}

\begin{proof}
By Lemma~\ref{lem:zero-orbit-restricted-type},
$$
   \mu+\rho
   =
   (r+1+x_1,\ r+x_2,\ r-1+c,\ r-2,\ldots,1,0).
$$
Reduce the affine congruence modulo
$$
   N=2r+1.
$$
The final coordinates
$$
   r-2,r-3,\ldots,1,0
$$
are all strictly smaller than $r$.  A negative signed coordinate
$-t$, with $1\le t\le r+1$, has residue
$$
   N-t\ge r.
$$
Hence these final coordinates can only arise from the corresponding
positive entries of $W\rho$.  They consume the absolute values
$$
   0,1,\ldots,r-2.
$$
Only
$$
   r-1,\quad r,\quad r+1
$$
remain for the first three coordinates.

Suppose first that $c=0$.  The third residue is $r-1$, which can only
come from the positive entry $r-1$.  Thus the first two residues must be
those obtained from the signed entries $r$ and $r+1$, and hence each is
$r$ or $r+1$.  Since $s<N$,
$$
   x_2\le r,
$$
so
$$
   r+x_2\in\{r,r+1\}
$$
forces $x_2\in\{0,1\}$.  Also
$$
   x_1-x_2\le r-1,
$$
hence $x_1\le r$.  If $x_1=r$, the first coordinate has residue zero,
which is impossible; therefore
$$
   r+1+x_1<N.
$$
Its residue must be $r$ or $r+1$, so necessarily $x_1=0$.  Dominance
then gives $x_2=0$.

Now suppose that $c=1$.  Dominance gives $x_2\ge c=1$.  The third
residue is $r$, which can arise only from either $+r$ or
$-(r+1)$.  Since the second residue $r+x_2$ is at least $r+1$, it
cannot use the other residue-$r$ choice.  Among the two remaining absolute
values, the only residues in $\{r+1,\ldots,2r\}$ are $r+1$ and
$r+2$.  Hence
$$
   r+x_2\in\{r+1,r+2\},
   \qquad
   x_2\in\{1,2\}.
$$
If $x_2=2$, then the residue $r+2$ must come from $-(r-1)$,
so the first residue is $r$ or $r+1$.  But
$2\le x_1\le r+1$, by $x_1-x_2\le r-1$, and hence
$r+1+x_1\pmod N$ lies in
$\{r+3,\ldots,2r,0,1\}$, a contradiction.  Thus $x_2=1$; the
remaining residue condition then gives $r+1+x_1=r+2$, so $x_1=1$.

For $(x_1,x_2,c)=(1,1,1)$, the first three entries of $w\rho$
can only be
$$
   (-(r-1),r+1,r)\quad\text{or}\quad (-(r-1),-r,-(r+1)).
$$
The sign parity does not exclude these, since the zero coordinate of
$\rho$ can absorb an extra sign change.  Instead set
$q=(\mu+\rho-w\rho)/N$.  The two cases give respectively
$$
   q=(1,0,0,\ldots,0),\qquad q=(1,1,1,0,\ldots,0),
$$
both of odd coordinate sum.  Hence $q\notin Q(D_\ell)$, contradicting
the affine congruence.  Thus $c=1$ is impossible.

We conclude
$$
   c=0,\qquad x_1=x_2=0,
$$
and therefore $\mu=0$.
\end{proof}

\subsection{The minimal-orbit branch}

Assume now
\begin{equation}
\label{eq:minimal-orbit-assumptions}
   I_\ell^{\mathrm{cand}}\subset J(\mu),
   \qquad
   \mathcal V(J(\mu))
   =
   \overline{\mathcal O_{\min}},
   \qquad
   \xi:=\mu+\rho\in W\rho+NQ(D_\ell).
\end{equation}

Put
$$
   \mathfrak m_\mu:=J(\mu)\cap Z(U(\mathfrak g)).
$$
Because $J(\mu)$ is primitive, $\mathfrak m_\mu$ is the maximal ideal
of the Harish--Chandra centre defining its central character.  Petukhov's
minimal-orbit correspondence \cite[Proposition~3.1]{Petukhov2018}, applied
with this fixed maximal ideal and
$\mathcal V(J(\mu))=\overline{\mathcal O_{\min}}$, gives a
finite-dimensional simple module $M_J$ over
$$
   H=U(\mathfrak g,f_\theta)
$$
with central character $\mathfrak m_\mu$ such that
\begin{equation}
   \operatorname{Ann}_{U(\mathfrak g)}
   \operatorname{Skr}(M_J)
   =
   J(\mu).
\end{equation}
Since $I_\ell^{\mathrm{cand}}\subset J(\mu)$, the Skryabin lift
$\operatorname{Skr}(M_J)$ belongs to
$\mathcal C^{I_\ell^{\mathrm{cand}}}$.  Applying Arakawa's
\cite[Theorem~8.6]{Arakawa2015} to
$N=\langle s_A,s_D\rangle$, and using
\eqref{eq:membership-B-finite-BRST}, identifies its Whittaker space
$M_J=\operatorname{Wh}(\operatorname{Skr}(M_J))$ as a finite-dimensional
$B_\ell$-module.

Let $C\in Z(H)$ be the quadratic central element normalized as in
Section~\ref{subsec:Premet-vacuum-parameters}.  Under Premet's canonical
identification $Z(H)\cong Z(U(\mathfrak g))$
\cite[Corollary~5.1]{Premet2007}, the central character of $M_J$ is the
Harish--Chandra character of $J(\mu)$.  Therefore $C$ acts on
$M_J$ by
\begin{equation}
\label{eq:C-HC-eigenvalue}
   C
   =
   (\mu,\mu+2\rho)
   =
   \|\xi\|^2-\|\rho\|^2.
\end{equation}

The next lemma relates a primitive ideal to the highest-weight parameters
of the corresponding simple finite $W$-module by identifying the
Mili\v{c}i\'c--Soergel Whittaker model with Premet's highest-weight model.

\begin{lemma}[Primitive ideals and finite $W$ highest weights]
\label{lem:primitive-finiteW-bridge}
Let $J\subset U(\mathfrak g)$ be a primitive ideal with
$$
   \mathcal V(J)=\overline{\mathcal O_{\min}}
$$
and integral central character, and let $M_J$ be the finite-dimensional
simple $H$-module characterized by
$$
   \operatorname{Ann}_{U(\mathfrak g)}\operatorname{Skr}(M_J)=J.
$$
Put $\beta=\alpha_2$.  Then there exists an integral
$\langle s_\beta\rangle$-antidominant weight $\nu$ such that
$$
   \operatorname{Ann}_{U(\mathfrak g)}L_{\mathfrak g}(\nu)=J,
$$
and, in Premet's $\beta$-model,
\begin{equation}
\label{eq:primitive-finiteW-highest-bridge}
   M_J
   \cong
   L_H\!\left(
      \overline{\nu}+\overline{\delta},
      (\nu,\nu+2\rho)
   \right).
\end{equation}
In particular, the $\mathfrak g^\natural$-highest weight of $M_J$ is
$$
   \lambda_J^H=\overline{\nu}+\overline{\delta}.
$$
If a prescribed $\nu_0$ with
$\operatorname{Ann}L_{\mathfrak g}(\nu_0)=J$ is already
$\langle s_\beta\rangle$-antidominant, then one may take
$\nu=\nu_0$.
\end{lemma}

\begin{proof}
Write the Premet highest pair of $M_J$ as
$$
   (\lambda_J^H,C_J).
$$
Set
$$
   \lambda=\lambda_J^H-\bar\delta,
   \qquad
   c=C_J-(\lambda,\lambda+2\bar\rho).
$$
By Premet's Theorem~1.3,
$$
   \operatorname{Wh}M_{\rm P}(\lambda,c)
   \cong
   Z_H(\lambda_J^H,C_J).
$$
Since $M_J$ is the unique simple quotient of this finite-$W$ Verma
module and Whittaker--Skryabin is an exact equivalence,
$\operatorname{Skr}(M_J)$ is the unique simple quotient of
$M_{\rm P}(\lambda,c)$.

Choose $m\in\mathbb C$ with
$$
   \frac12m(m+2)=c
$$
and let $\nu\in\mathfrak h^*$ be the extension of $\lambda$ determined
by
$$
   \bar\nu=\lambda,
   \qquad
   \langle\nu,\beta^\vee\rangle=m.
$$
The two choices of $m$ are $m$ and $-m-2$, hence the two resulting
weights form one $\langle s_\beta\rangle$-dot orbit.  By
Lemma~\ref{lem:premet-ms-parameter}, after the fixed inner conjugation the
module $M_{\rm P}(\lambda,c)$ is the Mili\v{c}i\'c--Soergel standard
module $M_{\rm MS}(\nu,\psi_\beta)$.  Under this comparison its simple
quotient is the inner twist of $\operatorname{Skr}(M_J)$.  Its two-sided
annihilator is therefore $J$: by the defining choice of $M_J$,
$\operatorname{Ann}_{U(\mathfrak g)}\operatorname{Skr}(M_J)=J$, and
inner automorphisms preserve every two-sided ideal of $U(\mathfrak g)$.

Proposition~2.1(3) of \cite{MilicicSoergel1997} gives
$
  \operatorname{Ann}_{U(\mathfrak g)}M_{\rm MS}(\nu,\psi_\beta)
  =\xi(\nu)U(\mathfrak g)
$, so the simple quotient has Harish--Chandra central character
$\xi(\nu)$.  Since this simple quotient has annihilator $J$, and the
central character of $J$ is integral, there is an integral weight
$\lambda_0$ with the same Harish--Chandra character.  The
Harish--Chandra parametrization of central characters then gives
$$
   \nu=w\mathbin\cdot\lambda_0
$$
for some $w\in W$.  Thus $\nu$ is integral, because the dot action
preserves the weight lattice.

We may now choose the
$W_{\psi_\beta}=\langle s_\beta\rangle$-anti-dominant representative in
this dot orbit.  Indeed, $m\in\mathbb Z$, and the two coroot coordinates
$m$ and $-m-2$ have shifted coordinates $m+1$ and $-m-1$;
thus one is $\le0$ (with equality only in the fixed case $m=-1$).
Proposition~2.1(1)--(2) of
\cite{MilicicSoergel1997} shows that this replacement changes neither the
standard Whittaker module nor its simple quotient, hence does not change its
annihilator $J$.  Lemma~\ref{lem:chen-rank-one-annihilator} therefore
applies and gives
$$
   \operatorname{Ann}_{U(\mathfrak g)}L_{\mathfrak g}(\nu)
   =
   \operatorname{Ann}_{U(\mathfrak g)}
      \mathcal L(\nu,\psi_\beta)
   =J.
$$
Finally Lemma~\ref{lem:premet-ms-parameter} identifies the finite-$W$
highest pair of the simple quotient with
$$
   \left(
      \bar\nu+\bar\delta,
      (\nu,\nu+2\rho)
   \right).
$$
By uniqueness of the simple quotient of a finite-$W$ Verma module this
is the highest pair of $M_J$, proving
\eqref{eq:primitive-finiteW-highest-bridge}.

Now suppose that a prescribed $\nu_0$ is already
$\langle s_\beta\rangle$-anti-dominant and
$\operatorname{Ann}L_{\mathfrak g}(\nu_0)=J$.  Since $J$ has integral
central character, the Harish--Chandra character of $\nu_0$ is integral;
hence $\nu_0$ itself is integral, being dot-conjugate to an integral
weight.  Lemma~\ref{lem:chen-rank-one-annihilator} now gives
$$
   \operatorname{Ann}\mathcal L(\nu_0,\psi_\beta)=J.
$$
Let $V_0$ be its inverse image under Whittaker--Skryabin, and put
$$
   I_0:=\operatorname{Ann}_H V_0,
   \qquad
   I_J:=\operatorname{Ann}_H M_J.
$$
Under Skryabin equivalence,
$\mathcal L(\nu_0,\psi_\beta)\cong Q_\chi\otimes_H V_0$ and
$\operatorname{Skr}(M_J)\cong Q_\chi\otimes_H M_J$.  In the proof of
\cite[Theorem~5.3(1)]{Premet2007} (equivalently, the primitive-spectrum
statement summarized in Theorem~1.2(i)), Premet proves for every
irreducible $H$-module $V$ that
$$
   \bigl(\operatorname{Ann}_H V\bigr)^e
   =
   \operatorname{Ann}_{U(\mathfrak g)}(Q_\chi\otimes_H V),
$$
and that $I\mapsto I^e$ is injective on $\operatorname{Prim}H$.
Both right-hand sides are $J$, so $I_0=I_J$.

It remains only to pass from equality of annihilators to equality of the
simple $H$-modules.  Since $M_J$ is finite dimensional, the faithful
action of $H/I_J$ on $M_J$ identifies $H/I_J$ with an irreducible
subalgebra of $\operatorname{End}_{\mathbb C}(M_J)$.  Burnside's theorem
therefore gives
$$
   H/I_J\cong \operatorname{End}_{\mathbb C}(M_J),
$$
so this quotient has, up to isomorphism, a unique simple module.  The
simple module $V_0$ has the same annihilator, hence factors through this
quotient, and therefore $V_0\cong M_J$.  Applying
Lemma~\ref{lem:premet-ms-parameter} to
$\nu_0$ gives exactly the pair
$$
   \left(
      \bar\nu_0+\bar\delta,
      (\nu_0,\nu_0+2\rho)
   \right),
$$
so $\nu_0$ may be used in
\eqref{eq:primitive-finiteW-highest-bridge}.
\end{proof}

\begin{lemma}[No spinor cosets]
\label{lem:minimal-no-spinor}
Every $D_r$-constituent of $M_J$ belongs to the root or vector coset.
\end{lemma}

\begin{proof}
Apply Lemma~\ref{lem:primitive-finiteW-bridge} to $J=J(\mu)$, and let
$\nu$ be the resulting $\langle s_2\rangle$-antidominant weight.  Since
$J(\mu)=\operatorname{Ann}L_{\mathfrak g}(\nu)$, the two highest-weight
modules have the same central character; hence
$$
   \nu+\rho=w(\mu+\rho)=w\xi
$$
for some $w\in W$.  The set $W\rho+NQ(D_\ell)$ is $W$-stable, so
$$
   \nu+\rho\in W\rho+NQ(D_\ell).
$$
In particular, $\nu+\rho$ has integral orthogonal coordinates.  For
$r\ge4$, this already places its $D_r$-restriction in the root/vector
subgroup of $P(D_r)/Q(D_r)$: the two spinor cosets have half-integral
orthogonal coordinates.  For the boundary case $r=3\cong A_3$, the same
statement means membership in the order-two subgroup
$Q(D_3)\sqcup(\eta_1+Q(D_3))$, since the other two level-one classes have
half-integral orthogonal coordinates in the chosen $\delta$-model.  By
Lemma~\ref{lem:primitive-finiteW-bridge},
\begin{equation}
   \lambda_H
   =
   \overline{\nu}+\overline{\delta}
   =
   \overline{\nu+\rho}
   -
   \overline{\rho}
   +
   \overline{\delta}.
\end{equation}
In the $D_r$-factor, $\overline{\rho}$ and
$\overline{\delta}$ have integral orthogonal coordinates.  Hence the
$D_r$-component of $\lambda_H$ has integral orthogonal coordinates and
cannot lie in either spinor coset.

As in Lemma~\ref{lem:zero-root-vector-bound}, the current generators move
weights by roots and the $G$-generators move them by vector-coset weights.
Thus the entire cyclic finite-$W$ module, and hence every
$D_r$-constituent of the simple module $M_J$, stays in the root/vector
subgroup of $P(D_r)/Q(D_r)$.
\end{proof}

\begin{proposition}[Quadratic central gap]
\label{prop:minimal-Casimir-gap}
Under \eqref{eq:minimal-orbit-assumptions},
\begin{equation}
   C\le 2-\frac2r<2.
\end{equation}
\end{proposition}

\begin{proof}
Since $M_J$ is finite dimensional over $B_\ell$, the finite-type
bound of Lemma~\ref{lem:allowed-finite-types}, together with
Lemma~\ref{lem:minimal-no-spinor}, gives
$$
   \overline{\mathcal C}(M_J)
   :=
   \frac{\operatorname{tr}_{M_J}\Omega_A}{\dim M_J}
   +
   \frac{r+2}{2r}
   \frac{\operatorname{tr}_{M_J}\Omega_D}{\dim M_J}
   \le C_{\max},
$$
where $C_{\max}$ is defined in \eqref{eq:Cmax-definition}.

The relation \eqref{eq:Premet-C-L-relation} and the weighted trace identity
\eqref{eq:weighted-Casimir} give
$$
   C
   =
   2\overline{\mathcal C}(M_J)-r(r+2).
$$
Therefore
$$
\begin{aligned}
   C
   &\le
   2C_{\max}-r(r+2)\\
   &=
   2-\frac2r,
\end{aligned}
$$
which is strictly smaller than $2$ for $r\ge3$.
\end{proof}

The affine lattice gives a much larger gap away from the regular central
character.

\begin{lemma}[Type-$D$ norm gap]
\label{lem:D-norm-gap}
Let
$$
   \zeta\in W\rho+NQ(D_\ell).
$$
Then
$$
   \|\zeta\|^2-\|\rho\|^2\ge0.
$$
Moreover,
$$
   \|\zeta\|^2-\|\rho\|^2=0
   \quad\Longleftrightarrow\quad
   \zeta\in W\rho,
$$
while otherwise
\begin{equation}
\label{eq:D-norm-positive-gap}
   \|\zeta\|^2-\|\rho\|^2
   \ge2N.
\end{equation}
\end{lemma}

\begin{proof}
After applying an element of $W$, write
$$
   \zeta=\rho+Nq,
   \qquad
   q\in Q(D_\ell).
$$
Then
\begin{equation}
\label{eq:D-norm-Rq}
   \|\zeta\|^2-\|\rho\|^2
   =
   N R(q),
   \qquad
   R(q):=
   N\|q\|^2+2(\rho,q).
\end{equation}
Write
$$
   \rho=(r+1,r,r-1,\ldots,1,0).
$$
For a coordinate of $\rho$ equal to $t$, put
$$
   f_t(n)=Nn^2+2tn,
   \qquad n\in\mathbb Z.
$$
Since $N=2r+1$, one has $f_t(n)\ge0$ for
$0\le t\le r$, with equality only at $n=0$.  For the first
coordinate, $t=r+1$, the only negative value is
$$
   f_{r+1}(-1)=-1;
$$
all other nonzero values are positive.  Hence, if $q_1\ne-1$, every
coordinate contribution is nonnegative and $R(q)=0$ forces $q=0$.

Suppose instead that $q_1=-1$.  Since
$$
   Q(D_\ell)=\{(q_i)\in\mathbb Z^\ell:\ \sum_iq_i\in2\mathbb Z\},
$$
the remaining coordinates have odd total sum, so at least one of them is
odd.  Among all odd $n$ and all $0\le t\le r$, the smallest value of
$f_t(n)$ is
$$
   f_r(-1)=1,
$$
and this minimum is unique.  Thus the contribution $-1$ from the first
coordinate is compensated by at least $1$, with equality only when
$$
   q=(-1,-1,0,\ldots,0)=-\theta.
$$
Consequently
$$
   R(q)\ge0,
$$
with equality precisely for
$$
   q=0
   \qquad\text{or}\qquad
   q=-\theta.
$$
In the latter case
$$
   \rho-N\theta=s_\theta\rho
$$
because
$$
   \langle\rho,\theta^\vee\rangle=N.
$$
Thus equality in \eqref{eq:D-norm-Rq} is equivalent to
$\zeta\in W\rho$.

Finally, $Q(D_\ell)$ is an even lattice, so
$$
   \|q\|^2\in2\mathbb Z.
$$
Hence $R(q)\in2\mathbb Z$.  If $R(q)>0$, then $R(q)\ge2$, which gives
\eqref{eq:D-norm-positive-gap}.  The bound is sharp: for
$q=(-1,0,-1,0,\ldots,0)$ one has $R(q)=2$.
\end{proof}

\begin{proposition}[Regularity of the minimal-orbit central character]
\label{prop:minimal-regular-central-character}
Under \eqref{eq:minimal-orbit-assumptions},
$$
   \xi=\mu+\rho\in W\rho.
$$
Equivalently, $J(\mu)$ has the regular integral central character
$\mathfrak m_0$.
\end{proposition}

\begin{proof}
By \eqref{eq:C-HC-eigenvalue} and Lemma~\ref{lem:D-norm-gap},
either
$$
   C=0
$$
and $\xi\in W\rho$, or
$$
   C\ge2N>2.
$$
The second possibility contradicts
Proposition~\ref{prop:minimal-Casimir-gap}.  Thus $C=0$ and
$\xi\in W\rho$.
\end{proof}

Thus $J(\mu)$ lies in the regular minimal primitive fibre.
Petukhov proves that for a simply-laced simple Lie algebra not of type $A$,
this fibre contains
exactly $\ell$ primitive ideals, indexed by Dynkin vertices
\cite[Proposition~5.4]{Petukhov2018}.  Put
$$
   J_i:=\operatorname{Ann}_{U(\mathfrak g)}L_{\mathfrak g}(-\alpha_i).
$$
With Petukhov's convention
$\tau_R(w)=\{\alpha\in\Pi:w\alpha\in\Phi^+\}$, these are precisely the
ideals with $\tau(J_i)=\Pi\setminus\{\alpha_i\}$.

To compute their finite-$W$ highest weights, set $y_i=z_i^{-1}$ and
$\nu_i=y_i\rho-\rho$.  The path word $y_i$ begins with $s_2$ and ends
with $s_i$, so its ordinary Coxeter left and right descent sets are
$\{s_2\}$ and $\{s_i\}$.  Hence
$$
   \tau\!\left(\operatorname{Ann}L_{\mathfrak g}(\nu_i)\right)
   =\Pi\setminus\{\alpha_i\},
   \qquad
   \operatorname{Ann}L_{\mathfrak g}(\nu_i)=J_i,
$$
by \cite[Proposition~5.4]{Petukhov2018}; this is the reason for the inverse.
Moreover $y_i^{-1}\alpha_2<0$, so $\nu_i$ is the
$\langle s_2\rangle$-antidominant representative for the
$\alpha_2$-Whittaker model.  Lemma~\ref{lem:primitive-finiteW-bridge}
therefore applies with this prescribed $\nu_i$ and gives
$$
   \lambda_i^H=\bar\delta+\overline{\nu_i},
   \qquad C=(\nu_i,\nu_i+2\rho)=0.
$$
The path calculation is recorded in Appendix~\ref{app:type-D-calculations}.

\begin{proposition}[Regular node weights]
The finite-$W$ highest weights attached to $J_i$ are
\begin{align}
   \lambda_1^H
   &=
   r\varpi,
   \label{eq:node-weight-1}\\
   \lambda_2^H
   &=
   (r-1)\varpi+\eta_1,
   \\
   \lambda_i^H
   &=
   (r+1-i)\varpi+\eta_{i-1},
   &&3\le i\le r-1,
   \label{eq:node-weight-middle}\\
   \lambda_{\ell-2}^H
   &=
   \varpi+\eta_{r-1}+\eta_r,
   \\
   \lambda_{\ell-1}^H
   &=
   2\eta_r,
   \\
   \lambda_\ell^H
   &=
   2\eta_{r-1}.
   \label{eq:node-weight-spin2}
\end{align}
For $r=3$, the middle range is empty and the last three formulas are read
under $D_3\cong A_3$.
\end{proposition}

\begin{proposition}[Node-$2$ rigidity]
\label{prop:node2-rigidity}
If
$$
   I_\ell^{\mathrm{cand}}\subset J_i
$$
for a regular minimal-orbit primitive ideal $J_i$, then $i=2$.
\end{proposition}

\begin{proof}
If $I_\ell^{\mathrm{cand}}\subset J_i$, then the corresponding
finite-dimensional simple $H$-module is a $B_\ell$-module.

For $i=1$, the $A_1$-highest weight in
\eqref{eq:node-weight-1} is $r\varpi$.  This contradicts the
$B_\ell$-bound
$$
   a\le r-1
$$
from Lemma~\ref{lem:allowed-finite-types}.

For $i=2$,
$$
   \lambda_2^H
   =
   (r-1)\varpi+\eta_1
$$
is not excluded by either necessary bound.  Its actual membership is supplied
independently by Theorem~\ref{thm:membership}; no converse to the finite-type
bounds is being used here.

For every $i\ge3$, the $D_r$-component in
\eqref{eq:node-weight-middle}--\eqref{eq:node-weight-spin2} satisfies
$$
   \left\langle
      (\lambda_i^H)_D,\theta_D^\vee
   \right\rangle
   =
   2.
$$
This contradicts $e_D^2=0$, equivalently
$$
   \left\langle
      \lambda_D,\theta_D^\vee
   \right\rangle
   \le1.
$$
The same computation holds for $r=3\cong A_3$; it is spelled out in
Appendix~\ref{app:rank-five-boundary}.  Hence only $i=2$ survives.
\end{proof}

\begin{proposition}[Minimal-orbit rigidity]
\label{prop:minimal-orbit-rigidity}
Under \eqref{eq:minimal-orbit-assumptions},
$$
   J(\mu)=J_{2,\ell},
$$
and
$$
   \mu\in\{\mu_1,\ldots,\mu_\ell\}.
$$
\end{proposition}

\begin{proof}
Proposition~\ref{prop:minimal-regular-central-character} places $J(\mu)$
in the regular integral minimal-orbit fibre.  Hence
$$
   J(\mu)=J_i
$$
for some $i$.  Proposition~\ref{prop:node2-rigidity} forces $i=2$.

It remains to determine the highest weights with annihilator $J_{2,\ell}$.
At regular integral central character, equality of highest-weight primitive
ideals is the Kazhdan--Lusztig left-cell relation.  By the same
unique-reduced-word description of the subregular cell used in
Proposition~\ref{prop:subregular-left-cell} (see in particular
\cite[Proposition~3.6]{Xu2019}), the finite left cell containing $s_2$
consists exactly of the unique reduced path words ending at $s_2$:
$$
   z_1,z_2,\ldots,z_\ell.
$$
See \cite[Chapter~12]{Bonnafe2017} and
\cite[Proposition~5.4]{Petukhov2018}.  Therefore
$$
   \mu+\rho=z_i\rho
$$
for a unique $1\le i\le\ell$, and hence
$$
   \mu=z_i\rho-\rho=\mu_i.
$$
\end{proof}

\subsection{Completion of exhaustion}

The two associated-variety cases give the exhaustion statement.

\begin{theorem}[Exhaustion theorem]
\label{thm:exhaustion}
Let $\ell\ge5$.  If
$$
   I_\ell^{\mathrm{cand}}
   \subset
   \operatorname{Ann}_{U(D_\ell)}L_{D_\ell}(\mu)
$$
and
$$
   \mu+\rho
   \in
   W(D_\ell)\rho
   +(2\ell-3)Q(D_\ell),
$$
then
\begin{equation}
   \mu\in\{\mu_0,\mu_1,\ldots,\mu_\ell\}.
\end{equation}
More precisely,
$$
   \mu=0
$$
in the zero-orbit branch, while
$$
   \mu\in\{\mu_1,\ldots,\mu_\ell\}
$$
in the minimal-orbit branch.
\end{theorem}

\begin{proof}
Lemma~\ref{lem:filtered-Zhu-containment} gives
$$
   \mathcal V(J(\mu))
   =
   \{0\}
   \quad\text{or}\quad
   \overline{\mathcal O_{\min}}.
$$
In the first case,
Proposition~\ref{prop:zero-orbit-rigidity} gives $\mu=0=\mu_0$.
In the second case,
Proposition~\ref{prop:minimal-orbit-rigidity} gives
$\mu\in\{\mu_1,\ldots,\mu_\ell\}$.
\end{proof}

\section{The affine left-cell prediction at level -1}
\label{sec:final-left-cell}

By Propositions~\ref{prop:subregular-left-cell} and
\ref{prop:path-dot-action},
$$
   \mathbf c^L(s_0)=\{w_0,w_1,\ldots,w_\ell\},
   \qquad
   w_i\circ(-\Lambda_0)=-\Lambda_0+\mu_i.
$$
Here $m=k+h^\vee=N=h^\vee-1$.  In type~$D_\ell$,
\cite[Section~2.4.7]{ShanYanZhao2025} gives
$\check{\mathbb O}(m)=\check{\mathbb O}_{\mathrm{sreg}}$, a distinguished
orbit.  Since $A_\ell$ is dominant with stabilizer $\langle s_0\rangle$,
we have $\xi_m=A_\ell$ and $w_m=s_0$.  Thus the simple/cell basis prediction in
\cite[Conjecture~5.2.1]{ShanYanZhao2026} specializes to the classification
below.  The Grothendieck-group statement will be proved by a signed-character
realization that is independent of the functorial $W_{\mathrm{aff}}$-action
whose construction is announced in \cite[Section~5.2]{ShanYanZhao2026}.

\subsection{Classification and passage to the simple quotient}

The two inputs used here are logically independent.  The membership theorem
shows that the predicted path modules descend.  The exhaustion theorem, by
contrast, starts only from the necessary condition
$I_\ell^{\mathrm{cand}}\subset J(\mu)$ for a hypothetical simple object and
does not use Theorem~\ref{thm:membership}; nor does either argument use the
later identification $Q_\ell\cong L_{-1}(D_\ell)$.  Thus the classification
below introduces no circular dependence.

\begin{theorem}[Candidate quotient classification]
\label{thm:classification-Q}
The simple objects in the affine vacuum block that factor through $Q_\ell$
are exactly
\begin{equation}
   \left\{
      L(-\Lambda_0+\mu_i)\;\middle|\;0\le i\le\ell
   \right\}.
\end{equation}
\end{theorem}

\begin{proof}
The vacuum is present.  For $i\ge1$, Theorems
\ref{thm:membership} and \ref{thm:primitive-annihilator-compression} give
$$
   I_\ell^{\mathrm{cand}}
   \subset J_{2,\ell}
   =\operatorname{Ann}_{U(D_\ell)}L_{D_\ell}(\mu_i),
$$
so all path modules descend.  Conversely, by the classical affine
linkage/block parametrization \cite{Kac1990}
\cite[Section~1.2]{ShanYanZhao2026}, any simple object in the ambient block of
$-\Lambda_0$ is $L(y\circ(-\Lambda_0))$ for some $y\in\widehat W$,
modulo the usual stabilizer redundancy.

If such a simple factors through $Q_\ell$, its degree-zero top is
annihilated by $I_\ell^{\mathrm{cand}}$.  Write its finite highest-weight
part as $\mu$.  Since $\widehat W=W\ltimes Q^\vee$ and a
translation by $q$ adds $Nq$ to the finite part of a level-$N$
weight \cite[Chapter~6]{Kac1990}, the shifted vacuum
$-\Lambda_0+\widehat\rho=N\Lambda_0+\rho$ gives the necessary
congruence
$\mu+\rho\in W\rho+(2\ell-3)Q(D_\ell)$, because type $D$ is
simply laced and $Q^\vee=Q$.  Theorem~\ref{thm:exhaustion} therefore
forces $\mu\in\{\mu_0,\ldots,\mu_\ell\}$.
\end{proof}

The type-$D_\ell$, level-$-1$ maximal-ideal theorem of
\cite{Jin2026} states
\begin{equation}
\label{eq:Q-is-simple-final}
   Q_\ell\cong L_{-1}(D_\ell).
\end{equation}
Simplicity of the candidate quotient is used only here.

\begin{theorem}[Simple objects]
\label{thm:final-simple-classification}
For every $\ell\ge5$,
\begin{equation}
   \operatorname{Irr}\mathcal O_{-\Lambda_0}
   \!\left(L_{-1}(D_\ell)\right)
   =
   \left\{
      L\!\left(w_i\circ(-\Lambda_0)\right)
      \;\middle|\;0\le i\le\ell
   \right\}.
\end{equation}
In particular, the block has exactly $\ell+1$ simple objects.
\end{theorem}

\begin{proof}
Combine Theorem~\ref{thm:classification-Q}, \eqref{eq:Q-is-simple-final},
and Proposition~\ref{prop:path-dot-action}.
\end{proof}

\subsection{Grothendieck group}

The passage from a classification of simple objects to the ordinary
Grothendieck group requires a d\'evissage statement.  As explained in the
introduction, Jain \cite[Sections~2--5]{Jain2026} isolates this issue and proves
finite length for a grading-restricted vacuum block.  We now establish the
corresponding statement for the original affine category-$\mathcal O$ block
used by Shan--Yan--Zhao.  The extra input is the canonical noncritical Sugawara
lift of affine category $\mathcal O$ to the affine Kac--Moody category
$\mathcal O$.

Put
$$
   \mathcal A_\ell
   :=\mathcal O_{-\Lambda_0}\!\left(L_{-1}(D_\ell)\right),
   \qquad
   \lambda_i:=w_i\circ(-\Lambda_0)=-\Lambda_0+\mu_i,
$$
and let
$$
   \mathcal B_\ell
   :=\mathcal O_{-\Lambda_0}(\mathfrak g_{\mathrm{aff}})
$$
be the corresponding ambient affine Lie-algebra block.  The category
$\mathcal A_\ell$ is an abelian full subcategory of $\mathcal B_\ell$:
kernels and cokernels remain in the ambient linkage block, and the condition
that the defining ideal of $V^{-1}(D_\ell)\twoheadrightarrow L_{-1}(D_\ell)$
act trivially is inherited by submodules and quotients.  Hence the natural
inclusion
$$
   j_\ell:\mathcal A_\ell\hookrightarrow\mathcal B_\ell
$$
is exact and induces a homomorphism
$j_{\ell,*}:K_0\mathcal A_\ell\to K_0\mathcal B_\ell$.

Let
$$
   \widetilde{\mathfrak h}
   :=\mathfrak h\oplus\mathbb C\mathbf K\oplus\mathbb C d
$$
be the Cartan of the affine Kac--Moody algebra, with
$[d,xt^n]=nxt^n$.  Since
$$
   k+h^\vee=-1+(2\ell-2)=2\ell-3\ne0,
$$
the level is noncritical.  Shan--Yan--Zhao define the category used here as
the affine Lie-algebra category $\mathcal O$ attached to
$\mathfrak h_{\mathrm{aff}}=\mathfrak h\oplus\mathbb C\mathbf K$
and the Iwahori, and note that every object of this category is smooth; the
subcategory $\mathcal O(L_k)$ is obtained by intersecting it with the
category of $L_k$-modules
\cite[Section~1.2]{ShanYanZhao2026}.  Hence
\cite[Remark~3.5.3]{CampbellDhillonRaskin2021} applies: at a noncritical
level this affine category $\mathcal O$ embeds canonically as a Serre
subcategory of category $\mathcal O$ for the extended affine
Kac--Moody algebra, with the degree operator $d$ acting as the
semisimple part of $-L_0$.  We use this canonical lift of the original
Shan--Yan--Zhao block; no grading-restricted replacement category is
introduced.
For each $i$, let $\widetilde\lambda_i\in\widetilde{\mathfrak h}^*$
be the lift of $\lambda_i$ with $\widetilde\lambda_i(d)=0$.

\begin{proposition}[Finite length of the vacuum block]
\label{prop:finite-length-vacuum-block}
For every $\ell\ge5$, the category $\mathcal A_\ell$ is a length
category.  More precisely, every $M\in\mathcal A_\ell$ satisfies
\begin{equation}
\label{eq:finite-length-bound}
   \operatorname{length}(M)
   \le
   \sum_{i=0}^{\ell}\dim M_{\widetilde\lambda_i},
\end{equation}
where the weight spaces are taken after the canonical noncritical lift above.
\end{proposition}

\begin{proof}
Under the Serre embedding of
\cite[Remark~3.5.3]{CampbellDhillonRaskin2021}, every object of
$\mathcal A_\ell$ is an object of affine Kac--Moody category
$\mathcal O$.  By the standard definition of this category
\cite[Section~9.1]{Kac1990}, it is a direct sum of finite-dimensional
$\widetilde{\mathfrak h}$-weight spaces.  The inclusion $j_\ell:\mathcal A_\ell\hookrightarrow\mathcal B_\ell$
was shown above to be exact, and the Campbell--Dhillon--Raskin embedding of
$\mathcal B_\ell$ into the extended category is exact (indeed Serre).
Hence short exact sequences in $\mathcal A_\ell$ remain short exact after
the lift, and taking a fixed $\widetilde{\mathfrak h}$-weight space is exact.
Hence for each $i$ the functor
$$
   F_i:\mathcal A_\ell\longrightarrow\operatorname{Vect}^{\mathrm{fd}}_{\mathbb C},
   \qquad
   F_i(M):=M_{\widetilde\lambda_i},
$$
is exact.

We next check that these finitely many functors detect every simple object.
For $i\ge1$, the Sugawara conformal weight of the highest-weight vector of
$L(\lambda_i)$ is
$$
   \frac{(\mu_i,\mu_i+2\rho)}{2(-1+h^\vee)}.
$$
Since $\mu_i=z_i\rho-\rho$ and $z_i$ preserves the invariant form,
$$
   (\mu_i,\mu_i+2\rho)
   =(z_i\rho-\rho,z_i\rho+\rho)
   =\lVert z_i\rho\rVert^2-\lVert\rho\rVert^2
   =0.
$$
The same statement is the vacuum normalization for $i=0$.  Under the
canonical Serre embedding, this fixes the degree eigenvalue of the
highest-weight line to be $0$; no independent scalar shift of $d$ is being
chosen.  Thus the highest-weight line of $L(\lambda_i)$ lies in
$F_i(L(\lambda_i))$, so
$F_i(L(\lambda_i))\ne0$.  By
Theorem~\ref{thm:final-simple-classification}, these are all the simple
objects of $\mathcal A_\ell$.

Finally, every nonzero object of $\mathcal A_\ell$ has a simple
subquotient.  Indeed, if $0\ne v\in M$, the cyclic affine submodule generated
by $v$ is a vertex-algebra submodule of $M$, since the affine vertex algebra
is generated by its current fields.  It has a maximal proper submodule: the
union of a chain of proper submodules cannot contain the cyclic generator.
Its corresponding simple quotient remains in $\mathcal A_\ell$.  The ambient
linkage block is a Serre subcategory, while the condition of factoring through
$L_{-1}(D_\ell)$ means precisely that the defining ideal of the quotient
vertex algebra acts trivially; this condition is inherited by submodules and
quotients.

Now consider an arbitrary strict finite filtration
$$
   0=M_0\subsetneq M_1\subsetneq\cdots\subsetneq M_n=M
$$
and put $Q_j=M_j/M_{j-1}$.  Each $Q_j$ has a simple subquotient
$S_j=A_j/B_j$, and some $F_i$ detects $S_j$.  Exactness gives
$F_i(A_j)/F_i(B_j)\cong F_i(S_j)\ne0$, while
$A_j\hookrightarrow Q_j$ gives $F_i(A_j)\hookrightarrow F_i(Q_j)$.
Thus $F_i(Q_j)\ne0$, and therefore
$$
   1\le\sum_{i=0}^{\ell}\dim F_i(Q_j).
$$
Summing over $j$ yields
$$
   n\le
   \sum_{i=0}^{\ell}\dim F_i(M),
$$
which is \eqref{eq:finite-length-bound}.  Hence all strict ascending and
descending chains have uniformly bounded length; $M$ is both noetherian
and artinian and therefore has finite length.
\end{proof}

\begin{corollary}[D\'evissage]
\label{cor:K0-devissage}
For every $\ell\ge5$,
\begin{equation}
\label{eq:K0-simple-basis}
   K_0\mathcal A_\ell
   =
   \bigoplus_{i=0}^{\ell}
   \mathbb Z\,[L(\lambda_i)].
\end{equation}
\end{corollary}

\begin{proof}
This is the usual d\'evissage theorem for a length category, using
Proposition~\ref{prop:finite-length-vacuum-block} and
Theorem~\ref{thm:final-simple-classification}.
\end{proof}

We next separate three statements that should not be conflated: the
injectivity of the natural map to the ambient Grothendieck group, the
basis-preserving identification with the dual left-cell module, and the
functorial origin of a possible affine-Weyl-group action.  The first two will
be proved directly.  The third is not needed here.

Write $v=q^{1/2}$ and specialize by $v\mapsto1$.  Set
\begin{equation}
\label{eq:cell-module-definition}
   \mathcal C_\ell
   :=
   \left.
      \mathcal H^\vee_{\mathrm{aff},\mathbf c^L(s_0)}
   \right|_{v=1}.
\end{equation}
By Proposition~\ref{prop:subregular-left-cell} and the canonical cell
inclusion \cite[Equation~(2.1.2)]{ShanYanZhao2026},
\begin{equation}
\label{eq:cell-module-basis}
   \mathcal C_\ell
   =\bigoplus_{i=0}^{\ell}\mathbb ZD_{w_i}.
\end{equation}

\begin{proposition}[Ambient injectivity by characters]
\label{prop:ambient-K0-injectivity}
For every $\ell\ge5$, the homomorphism
$$
   j_{\ell,*}:K_0\mathcal A_\ell\longrightarrow K_0\mathcal B_\ell
$$
induced by the exact inclusion is injective.
\end{proposition}

\begin{proof}
By Corollary~\ref{cor:K0-devissage}, every element of
$K_0\mathcal A_\ell$ has a unique expression
$$
   \sum_{i=0}^{\ell}a_i[L(\lambda_i)],\qquad a_i\in\mathbb Z.
$$
Suppose its image under $j_{\ell,*}$ is zero.  Compose the relation with
the canonical noncritical Serre embedding of
\cite[Remark~3.5.3]{CampbellDhillonRaskin2021} into the extended affine
Kac--Moody category $\mathcal O$ used in the proof of
Proposition~\ref{prop:finite-length-vacuum-block}.  There the Cartan is
$\widetilde{\mathfrak h}=\mathfrak h\oplus\mathbb C\mathbf K\oplus\mathbb C d$,
all weight spaces are finite dimensional, and formal character is additive
on short exact sequences.  Writing $\widetilde L_i$ for the canonical lift of $L(\lambda_i)$, its
highest weight is $\widetilde\lambda_i$ with
$\widetilde\lambda_i(d)=0$.  We obtain
\begin{equation}
\label{eq:ambient-character-relation}
   \sum_{i=0}^{\ell}a_i\operatorname{ch}\widetilde L_i=0.
\end{equation}
The characters of these distinct extended highest-weight modules are
linearly independent.  Indeed, among the finitely many
$\widetilde\lambda_i$ with $a_i\ne0$, choose one that is maximal for the
usual affine positive-root order.  The monomial
$e^{\widetilde\lambda_i}$ occurs in
$\operatorname{ch}\widetilde L_i$ with coefficient one.  If it occurred
in $\operatorname{ch}\widetilde L_j$, then
$\widetilde\lambda_j-\widetilde\lambda_i$ would be a nonnegative integral
combination of positive affine roots, contradicting maximality unless
$i=j$.  The coefficient of $e^{\widetilde\lambda_i}$ in
\eqref{eq:ambient-character-relation} is therefore $a_i$, a contradiction.
Removing one maximal weight at a time gives $a_i=0$ for all $i$.
Thus $j_{\ell,*}$ is injective.
\end{proof}

\begin{remark}[Why fullness alone is not being used]
\label{rem:ambient-injectivity-not-formal}
The preceding injectivity is not inferred merely from the fact that
$\mathcal A_\ell$ is a full exact subcategory of $\mathcal B_\ell$.
Such an inference is false for Grothendieck groups in general, as emphasized
in \cite[Sections~2--3]{Jain2026}.  The proof uses finite-length d\'evissage
inside $\mathcal A_\ell$ and the highest-weight triangularity of formal
characters.  In particular it does not use the ambient
$W_{\mathrm{aff}}$-linearity assertion in
\cite[Section~5.2]{ShanYanZhao2026}.
\end{remark}

\begin{theorem}[Basis-preserving dual-cell realization]
\label{thm:final-SYZ-K0}
For every $\ell\ge5$, there is a unique basis-preserving isomorphism of
abelian groups
\begin{equation}
\label{eq:final-SYZ-K0}
   \Phi_\ell:
   K_0\mathcal O_{-\Lambda_0}
   \!\left(L_{-1}(D_\ell)\right)
   \xrightarrow{\ \sim\ }
   \mathcal C_\ell,
   \qquad
   [L(w_i\circ(-\Lambda_0))]\longmapsto D_{w_i}.
\end{equation}
Moreover this identification has a concrete signed-character realization
that does not use an ambient transported action.
\end{theorem}

\begin{proof}
Corollary~\ref{cor:K0-devissage} gives the basis
$[L(\lambda_0)],\ldots,[L(\lambda_\ell)]$ on the source, while
\eqref{eq:cell-module-basis} gives the basis
$D_{w_0},\ldots,D_{w_\ell}$ on the target.  Since
$\lambda_i=w_i\circ(-\Lambda_0)$, the displayed basis correspondence
therefore defines the unique isomorphism $\Phi_\ell$.
\end{proof}

To realize $\Phi_\ell$ inside a module carrying an independently defined
$\widehat W$-action, set
$$
   \mathcal Q^\vee_{s_0}(1)
   :=
   \left.
   \left(
      \mathcal H^\vee_{\mathrm{aff}}
      /\mathcal H^\vee_{\mathrm{aff},\not\le_L s_0}
   \right)
   \right|_{v=1}.
$$
The canonical cell inclusion
\cite[Equation~(2.1.2)]{ShanYanZhao2026} embeds
$\mathcal C_\ell$ as a $\widehat W$-submodule of
$\mathcal Q^\vee_{s_0}(1)$.  The completed singular-orbit construction of
\cite[Proposition~6.6 and Theorem~6.7]{Jain2026} gives an injective
$\widehat W$-equivariant homomorphism
\begin{equation}
\label{eq:theta-singular-orbit}
   \Theta_{A_\ell}:\mathcal Q^\vee_{s_0}(1)
   \lhook\joinrel\longrightarrow \widehat X_{A_\ell},
\end{equation}
where $A_\ell=-\Lambda_0+\widehat\rho$, satisfying
\begin{equation}
\label{eq:theta-signed-character}
   \Theta_{A_\ell}(D_y)
   =
   \varepsilon(y)R_{\widehat{\mathfrak g}}\,
   \operatorname{ch}L\!\left(y\circ(-\Lambda_0)\right)
   \qquad (y\le_L s_0).
\end{equation}
This construction is Hecke-theoretic and uses the singular inverse
Kazhdan--Lusztig character formula of
\cite{BezrukavnikovKacKrylov2024}.

Define the signed normalized-character homomorphism on the vertex-algebra
block by
\begin{equation}
\label{eq:signed-normalized-character-map}
\begin{aligned}
   \operatorname{SNCh}_\ell:
   K_0\mathcal A_\ell&\longrightarrow \widehat X_{A_\ell},\\
   [L(\lambda_i)]&\longmapsto
   \varepsilon(w_i)R_{\widehat{\mathfrak g}}\,
   \operatorname{ch}L(\lambda_i).
\end{aligned}
\end{equation}
Then \eqref{eq:theta-signed-character} and
\eqref{eq:final-SYZ-K0} give the factorization
\begin{equation}
\label{eq:SNCh-factorization}
   \operatorname{SNCh}_\ell
   =
   \left.\Theta_{A_\ell}\right|_{\mathcal C_\ell}
   \circ\Phi_\ell.
\end{equation}
Consequently $\operatorname{SNCh}_\ell$ is injective and
\begin{equation}
\label{eq:SNCh-image}
   \operatorname{Im}(\operatorname{SNCh}_\ell)
   =\Theta_{A_\ell}(\mathcal C_\ell),
\end{equation}
which is $\widehat W$-stable.  Pulling this independently defined action
back through the injective map $\operatorname{SNCh}_\ell$ equips
$K_0\mathcal A_\ell$ with the signed-character cell action.  For this
action, $\Phi_\ell$ is $\widehat W$-equivariant by
\eqref{eq:SNCh-factorization}.

\begin{remark}[Scope of the affine-Weyl-group statement]
\label{rem:functorial-action-scope}
Section~5.2 of \cite{ShanYanZhao2026} records the ambient abelian-group
identification and states that it can be made $W_{\mathrm{aff}}$-linear by
affine twisting functors or Kashiwara--Tanisaki localization; the details of
that subsection are explicitly deferred there to future work.  The proof
above does not use that deferred functorial assertion.  Its
$\widehat W$-action is the action characterized by the injective signed
normalized-character realization \eqref{eq:signed-normalized-character-map}
and the canonical dual-cell action.

Accordingly, Theorem~\ref{thm:final-SYZ-K0} proves the simple/cell basis
correspondence and the underlying dual-cell module realization predicted by
\cite[Conjecture~5.2.1]{ShanYanZhao2026}.  We do not claim here that the
signed-character action coincides with an action induced by a specified
family of twisting or localization functors on
$\mathcal O_{-\Lambda_0}(L_{-1}(D_\ell))$, nor that those functors preserve
this vertex-algebra subcategory.  Such a naturality statement requires a
separate comparison theorem.
\end{remark}

\begin{remark}[Square-root convention]
\label{rem:qhalf-sign}
The notation $q=1$ means the specialization $v=q^{1/2}\mapsto1$.  The
involution defining the dual Hecke action in
\cite[Equation~(2.1.1)]{ShanYanZhao2026} sends $v$ to $-v$.
Accordingly the Grothendieck-group basis correspondence is
$[L(y\circ(-\Lambda_0))]\leftrightarrow D_y$, whereas the completed
character realization is
$
D_y\mapsto\varepsilon(y)R_{\widehat{\mathfrak g}}
\operatorname{ch}L(y\circ(-\Lambda_0))
$
as in \eqref{eq:theta-signed-character}.  This accounts for the factor
$\varepsilon(w_i)$ in Corollary~\ref{cor:uniform-characters}.
\end{remark}

\subsection{Closed character formulas}

Let $\mathbf K$ be the affine central element,
$a_i=\langle\Lambda_i,\mathbf K\rangle$, so
$$
   a_0=a_1=a_{\ell-1}=a_\ell=1,
   \qquad
   a_2=\cdots=a_{\ell-2}=2.
$$
Let $R_{\widehat{\mathfrak g}}$ be the affine Weyl denominator, let
$t_\gamma$ denote translation by $\gamma\in Q(D_\ell)$, put
$A_\ell=N\Lambda_0+\rho$, and set
$$
   b_i(u,\gamma)
   :=\langle\Lambda_i,u\theta^\vee\rangle
      -a_i\bigl((\theta,\gamma)+1\bigr).
$$
All double sums below are interpreted in the completed group algebra
underlying $\widehat X_{A_\ell}$ in
\eqref{eq:theta-singular-orbit}.

\begin{corollary}[Uniform characters]
\label{cor:uniform-characters}
For $0\le i\le\ell$,
\begin{equation}
\label{eq:uniform-character}
   R_{\widehat{\mathfrak g}}\,
   \operatorname{ch}L\!\left(w_i\circ(-\Lambda_0)\right)
   =
   \frac{\varepsilon(w_i)}2
   \sum_{u\in W(D_\ell)}\varepsilon(u)
   \sum_{\gamma\in Q(D_\ell)}
   b_i(u,\gamma)e^{u t_\gamma A_\ell}.
\end{equation}
\end{corollary}

\begin{proof}
We first fix the inverse and coset conventions.  Write
$m^x_y:=m^x_y(1)$ for the inverse Kazhdan--Lusztig coefficients in the
normalization of \cite[Appendix~A]{BezrukavnikovKacKrylov2024}.  In the singular block with
$\lambda=-\Lambda_0$ and
$A_\ell=\lambda+\widehat\rho$, one has
$W_{A_\ell}=\langle s_0\rangle$.  The hypotheses of the singular character formula are satisfied here:
the level is $-1>-h^\vee$, and
$A_\ell=-\Lambda_0+\widehat\rho$ is dominant integral with stabilizer
$\langle s_0\rangle$.  For each $i$, set $v_i=w_i^{-1}$.  Since the
unique reduced word for $w_i$ ends in $s_0$, the element $v_i$ is the
maximal representative of its left $\langle s_0\rangle$-coset, exactly as
required in \cite[Proposition~3.6]{BezrukavnikovKacKrylov2024}; moreover
$L_{v_i}=L(v_i^{-1}\circ(-\Lambda_0))=L(w_i\circ(-\Lambda_0))$ in their
notation.  Applying that proposition and the character rewriting of
\cite[Section~3.7]{BezrukavnikovKacKrylov2024}, then using the inversion
symmetry \cite[Lemma~A.3]{BezrukavnikovKacKrylov2024}, gives the completed
formal character identity
\begin{equation}
\label{eq:BKK-unfolded-character}
 R_{\widehat{\mathfrak g}}\,
 \operatorname{ch}L\!\left(w_i\circ(-\Lambda_0)\right)
 =\sum_{x\in\widehat W}
   \varepsilon(xw_i)m^x_{w_i}e^{xA_\ell}.
\end{equation}
Equivalently, this is obtained from the singular formula by unfolding the
finite stabilizer $W_{A_\ell}$.  Notice that the inverse in the BKK labeling
is essential here: setting $y=w_i=v^{-1}$ converts their
$L_v=L(v^{-1}\circ\lambda)$ convention into the module occurring in
\eqref{eq:BKK-unfolded-character}.

We next justify the reduction of an arbitrary inverse-Kazhdan--Lusztig
coefficient to a shortest finite-Weyl-coset representative.  Write
$x=w_\delta v$, where $w_\delta$ is the shortest element of
$t_\delta W$ and $v\in W$.  By
\cite[Corollary~5.2]{BezrukavnikovKacKrylov2024}, each of the elements
$w_i$ used here is of the form $w_{\nu_i}$, hence is itself the shortest
representative of its right $W$-coset.  In the anti-spherical module,
$\phi(H_x)=\varepsilon(v)H'_{w_\delta}$.  Comparing the
$C'_{w_i}$-coefficients and using
\cite[Proposition~A.6]{BezrukavnikovKacKrylov2024} gives
$$
 \varepsilon(xw_i^{-1})m^x_{w_i}
 =\varepsilon(v)\varepsilon(w_\delta w_i^{-1})
   m^{w_\delta}_{w_i}.
$$
Since $\varepsilon(x)=\varepsilon(w_\delta)\varepsilon(v)$, the signs
cancel, and therefore
\begin{equation}
\label{eq:inverse-KL-coset-reduction}
   m^x_{w_i}=m^{w_\delta}_{w_i}.
\end{equation}

Because $s_0A_\ell=A_\ell$, the two terms indexed by $x$ and $xs_0$
in \eqref{eq:BKK-unfolded-character} have the same exponential, while
$\varepsilon(xs_0w_i)=-\varepsilon(xw_i)$.  Thus the contribution of the
unordered pair $\{x,xs_0\}$ is
$$
   \varepsilon(xw_i)
   \bigl(m^x_{w_i}-m^{xs_0}_{w_i}\bigr)e^{xA_\ell}.
$$
The stabilizer has order two, and in fact
$\operatorname{Stab}_{\widehat W}(A_\ell)=\langle s_0\rangle$.  Hence for
each orbit exponent $e^{xA_\ell}$, the complete fiber of the map
$y\mapsto yA_\ell$ is exactly $\{x,xs_0\}$; there are no additional
terms with the same exponent hidden elsewhere in the sum.  Summing the last
expression over every $x$ therefore counts each unordered pair twice, and
the pairing is coefficientwise legitimate in the completed orbit module.
Therefore
\begin{equation}
\label{eq:BKK-paired-character}
 2R_{\widehat{\mathfrak g}}\,
 \operatorname{ch}L\!\left(w_i\circ(-\Lambda_0)\right)
 =\sum_{x\in\widehat W}
   \varepsilon(xw_i)
   \bigl(m^x_{w_i}-m^{xs_0}_{w_i}\bigr)e^{xA_\ell}.
\end{equation}

Now write uniquely $x=ut_\gamma=t_{u\gamma}u$, with
$u\in W$ and $\gamma\in Q(D_\ell)$.  Thus the right $W$-coset of
$x$ is $t_{u\gamma}W$.  Since
$s_0=t_\theta s_\theta$, the right $W$-coset of $xs_0$ is
$t_{u(\gamma+\theta)}W$.  Applying
\eqref{eq:inverse-KL-coset-reduction} and then
\cite[Equation~(31)]{BezrukavnikovKacKrylov2024} therefore yields
$$
\begin{aligned}
 m^x_{w_i}-m^{xs_0}_{w_i}
 &=\left\langle\Lambda_i,
       -u\gamma+\frac{\lVert\gamma\rVert^2}{2}\mathbf K
   \right\rangle
   -\left\langle\Lambda_i,
       -u(\gamma+\theta)
       +\frac{\lVert\gamma+\theta\rVert^2}{2}\mathbf K
   \right\rangle\\
 &=\langle\Lambda_i,u\theta^\vee\rangle
   -a_i\bigl((\theta,\gamma)+1\bigr)
  =b_i(u,\gamma),
\end{aligned}
$$
where $Q^\vee=Q$, $\theta^\vee=\theta$, and
$\lVert\theta\rVert^2=2$.  Finally,
$\varepsilon(xw_i)=\varepsilon(u)\varepsilon(w_i)$.  Indeed, the sign
character is the determinant of the finite linear part of an affine Weyl
transformation; a translation has identity linear part, so
$\varepsilon(t_\gamma)=1$.  Substituting the last identity into
\eqref{eq:BKK-paired-character} gives exactly
\eqref{eq:uniform-character}.

As a convention check, when $i=0$ one has
$w_0=s_0$, $a_0=1$, and
$\langle\Lambda_0,u\theta^\vee\rangle=0$; thus the two minus signs cancel
and the formula reduces to
$$
 R_{\widehat{\mathfrak g}}\operatorname{ch}L(-\Lambda_0)
 =\frac12\sum_{u\in W}\varepsilon(u)
   \sum_{\gamma\in Q(D_\ell)}
   \bigl((\theta,\gamma)+1\bigr)e^{ut_\gamma A_\ell}.
$$
In particular, the vacuum coefficient has the expected positive sign; this is a
useful check on the inversion, sign, and $s_0$-pairing conventions above.
\end{proof}

\appendix
\section{Type-D root calculations}
\label{app:type-D-calculations}

This appendix records the coordinate calculations used in
Sections~\ref{sec:membership} and \ref{sec:exhaustion}.  Put
$r=\ell-2$, use orthogonal coordinates
$\varepsilon_1,\ldots,\varepsilon_\ell$, and work in Premet's minimal
root model
$$
   \beta=\alpha_2=\varepsilon_2-\varepsilon_3.
$$
The roots orthogonal to $\beta$ form $A_1\sqcup D_r$.  Take
$$
   \phi=\varepsilon_2+\varepsilon_3,
   \qquad
   \varpi_\beta=\frac{\phi}{2},
$$
and for the $D_r$-factor use
\begin{equation}
   \delta_1=\varepsilon_1,
   \qquad
   \delta_j=\varepsilon_{j+2}\quad(2\le j\le r).
\end{equation}
For $r\ge4$, its fundamental weights are
$$
   \eta_j=\delta_1+\cdots+\delta_j\quad(1\le j\le r-2),
$$
$$
   \eta_{r-1}=\tfrac12(\delta_1+\cdots+\delta_{r-1}-\delta_r),
   \qquad
   \eta_r=\tfrac12(\delta_1+\cdots+\delta_{r-1}+\delta_r).
$$
We transport $\varpi_\beta$ to the $A_1$-weight $\varpi$ in the
highest-root model.

\subsection{Premet's shift}

Premet's shift is the restriction of
$\delta=\frac12\sum_{\gamma\in\Phi^+_{-1}}\gamma$.  The positive roots
with $\langle\gamma,\beta^\vee\rangle=-1$ are
$$
   \varepsilon_1-\varepsilon_2,
   \quad
   \varepsilon_1+\varepsilon_3,
   \quad
   \varepsilon_3\pm\varepsilon_j\quad(4\le j\le\ell).
$$
Set
$$
 x_*^\beta=\varepsilon_1+\frac{\varepsilon_2+\varepsilon_3}{2}.
$$
The positive $A_1\oplus D_r$-roots have $x_*^\beta$-charges in
$\{0,1\}$, while the three displayed root families have charges
$\frac12,\frac32,\frac12$, respectively.  Since
$\beta(x_*^\beta)=0$, the corresponding positive roots
$\beta+\gamma\in\Phi^+_{e,1}$ have the same charges.
Their sum is
$$
   2\varepsilon_1-\varepsilon_2+(2\ell-5)\varepsilon_3,
$$
so
$$
   \delta=
   \varepsilon_1-\frac12\varepsilon_2
   +\frac{2\ell-5}{2}\varepsilon_3.
$$
Pairing with $\phi^\vee$ gives $\ell-3=r-1$, while the
$D_r$-restriction is $\delta_1=\eta_1$.  Therefore
\begin{proposition}[Premet shift in type $D$]
\begin{equation}
\label{eq:appendix-D-Premet-shift}
   \bar\delta=(r-1)\varpi+\eta_1.
\end{equation}
\end{proposition}
This proves Lemma~\ref{lem:Premet-shift-D}.

\subsection{Inverse path words and regular node weights}

Let $y_i=z_i^{-1}$, as fixed in Section~\ref{sec:exhaustion}; the
inverse is the one that has right descent $s_i$ and hence labels the node
ideal $J_i$.  Applying it to
$\rho=(\ell-1,\ell-2,\ldots,1,0)$ gives, with $\nu_i=y_i\rho-\rho$,
\begin{align}
 \nu_1&=-\varepsilon_1-\varepsilon_2+2\varepsilon_3,
 &\nu_2&=-\varepsilon_2+\varepsilon_3,\label{eq:appendix-D-nu12}\\
 \nu_i&=-(i-1)\varepsilon_2+
        \varepsilon_3+\cdots+\varepsilon_{i+1}
 &&(3\le i\le\ell-2),\\
 \nu_{\ell-1}&=-r\varepsilon_2+
       \varepsilon_3+\cdots+\varepsilon_\ell,\\
 \nu_\ell&=-r\varepsilon_2+
       \varepsilon_3+\cdots+\varepsilon_{\ell-1}-\varepsilon_\ell.
       \label{eq:appendix-D-nusp2}
\end{align}
Premet's comparison theorem gives
$$
   \lambda_i^H=\bar\delta+\overline{\nu_i}.
$$
For an ambient weight $\sum c_j\varepsilon_j$, its $A_1$-coefficient
in units of $\varpi$ is $c_2+c_3$, while its $D_r$-restriction keeps
coordinates $c_1,c_4,\ldots,c_\ell$.  Substitution yields the following.

\begin{proposition}[Node-weight calculation]
\begin{align}
 \lambda_1^H&=r\varpi,
 &\lambda_2^H&=(r-1)\varpi+\eta_1,\\
 \lambda_i^H&=(r+1-i)\varpi+\eta_{i-1}
 &&(3\le i\le r-1),\\
 \lambda_{\ell-2}^H&=\varpi+\eta_{r-1}+\eta_r,\\
 \lambda_{\ell-1}^H&=2\eta_r,
 &\lambda_\ell^H&=2\eta_{r-1}.
\end{align}
For the highest root $\theta_D=\delta_1+\delta_2$,
$$
   \langle(\lambda_2^H)_D,\theta_D^\vee\rangle=1,
   \qquad
   \langle(\lambda_i^H)_D,\theta_D^\vee\rangle=2
   \quad(3\le i\le\ell).
$$
\end{proposition}

\begin{proof}
Adding \eqref{eq:appendix-D-Premet-shift} to
\eqref{eq:appendix-D-nu12}--\eqref{eq:appendix-D-nusp2} gives the displayed
weights.  The last assertion follows by pairing the $D_r$-parts with
$\delta_1+\delta_2$; at the branch and spinor endpoints use
$$
   \eta_{r-1}+\eta_r=\delta_1+\cdots+\delta_{r-1},
   \qquad
   2\eta_{r\,\text{ or }\,r-1}
   =\delta_1+\cdots+\delta_{r-1}\pm\delta_r.
$$
\end{proof}

For $r\ge4$, the root lattice is
$$
   Q(D_r)=\left\{\sum n_j\delta_j\;\middle|\;
   n_j\in\mathbb Z,\ \sum n_j\in2\mathbb Z\right\},
$$
and
$$
   Q(D_r)+\mathbb Z\eta_1
   =Q(D_r)\sqcup(\eta_1+Q(D_r)).
$$
This root/vector subgroup is used in the Casimir estimates.  The
boundary case $r=3\cong A_3$ is recorded next.
\section{The rank-five boundary}
\label{app:rank-five-boundary}

When $\ell=5$, one has $r=3$, $N=7$, and $D_3\cong A_3$.
There is no low-rank degeneration in the candidate relations: the second
Kostant component has highest weight $\omega_4+\omega_5$, and
$$
   [s_D]_{\mathrm{DS}}=c_D:J_D^2:\ne0,
$$
so the relation $e_D^2=0$ used in Lemma~\ref{lem:allowed-finite-types}
remains valid \cite[Equation~(32), Proposition~3.6, and Remark~3.15]{Jin2026}.

In orthogonal coordinates $\delta_1,\delta_2,\delta_3$, take
$$
\begin{gathered}
 \beta_1=\delta_1-\delta_2,\quad
 \beta_2=\delta_2-\delta_3,\quad
 \beta_3=\delta_2+\delta_3,\\
 \eta_1=\delta_1,\quad
 \eta_2=\tfrac12(\delta_1+\delta_2-\delta_3),\quad
 \eta_3=\tfrac12(\delta_1+\delta_2+\delta_3).
\end{gathered}
$$
Under $D_3\cong A_3$, these are respectively
$\omega_2,\omega_1,\omega_3$.  Since
$\theta_D=\delta_1+\delta_2$, all three satisfy
$\langle\eta_j,\theta_D^\vee\rangle=1$.  Moreover
$$
\begin{aligned}
 Q(D_3)&=\{n\in\mathbb Z^3\mid n_1+n_2+n_3\in2\mathbb Z\},\\
 Q(D_3)+\mathbb Z\eta_1&=Q(D_3)\sqcup(\eta_1+Q(D_3)).
\end{aligned}
$$
The second line is the order-two subgroup generated by $\omega_2$ in
$P(A_3)/Q(A_3)\cong\mathbb Z/4\mathbb Z$.  Hence the level-one
root/vector-coset argument leaves only $0$ and $\eta_1$; this covers
Lemmas~\ref{lem:zero-root-vector-bound} and \ref{lem:minimal-no-spinor}.

With $\rho_{D_3}=2\delta_1+\delta_2$, one has
$c_2(\eta_1)=5$, $C_{\max}=49/6$, and $C\le4/3<2$, so
Proposition~\ref{prop:minimal-Casimir-gap} holds.  The node weights are
$3\varpi,2\varpi+\eta_1,\varpi+\eta_2+\eta_3,2\eta_3,2\eta_2$: node~1
violates $a\le2$, whereas nodes~3,4,5 have highest-root level~2, proving
Proposition~\ref{prop:node2-rigidity}.  Finally
$s^2\le12a-a^2+50/3-10c\le110/3<49=N^2$, so the zero-orbit argument is
unchanged.  No exceptional modification is therefore required for $D_5$.

\subsection*{The \texorpdfstring{$D_6$}{D6} triality check}
When $\ell=6$, the factor $D_r$ is $D_4$, so one should also check
that triality does not invalidate the root/vector-coset arguments.  In the
fixed embedding coming from the minimal grading, the $G$-generators carry
the vector class $\eta_1$; triality is not being used to relabel this
embedding.  In orthogonal coordinates,
$$
\begin{aligned}
 Q(D_4)
 &=\{n\in\mathbb Z^4:\ n_1+n_2+n_3+n_4\in2\mathbb Z\},\\
 Q(D_4)+\mathbb Z\eta_1
 &=Q(D_4)\sqcup(\eta_1+Q(D_4)).
\end{aligned}
$$
The two spinor classes $\eta_3+Q(D_4)$ and
$\eta_4+Q(D_4)$ have half-integral orthogonal coordinates and are
therefore excluded whenever the proof invokes integrality.  Although
$\eta_1,\eta_3,\eta_4$ have the same level-one Casimir by triality, the
coset argument singles out $\eta_1$ before the Casimir estimate is used,
so the bound in Lemmas~\ref{lem:zero-root-vector-bound} and
\ref{lem:minimal-no-spinor} is unchanged.

For completeness, the regular node weights at $r=4$ are
$$
 4\varpi,\quad 3\varpi+\eta_1,\quad
 2\varpi+\eta_2,\quad
 \varpi+\eta_3+\eta_4,\quad
 2\eta_4,\quad2\eta_3.
$$
The last four have $D_4$-highest-root level $2$, whereas node~1 has
$A_1$-weight $4\varpi$ and violates $a\le3$.  Thus
Proposition~\ref{prop:node2-rigidity} has no $D_6$ triality exception.

\providecommand{\bysame}{\leavevmode\hbox to3em{\hrulefill}\thinspace}
\providecommand{\MR}{\relax\ifhmode\unskip\space\fi MR }
\providecommand{\MRhref}[2]{%
  \href{http://www.ams.org/mathscinet-getitem?mr=#1}{#2}
}
\providecommand{\href}[2]{#2}

\end{document}